\documentclass[10 pt,leqno]{amsart}

\usepackage {amsfonts}
\usepackage{amsthm}
\usepackage{amssymb}
\usepackage{latexsym}
\usepackage{amsmath}
\usepackage{mathrsfs}
\usepackage{comment}

\newenvironment{rlist}
{

\begin{enumerate}}
{\end{enumerate}}

\theoremstyle{plain}
\newtheorem{tw}{Theorem} [section]

\newtheorem {lem} [tw]{Lemma}
\newtheorem {prop}[tw] {Proposition}
\newtheorem {propn}[tw] {Proposition}

\newtheorem{cor}[tw]{Corollary}

\theoremstyle{definition}
\newtheorem {deft}[tw] {Definition}
\newtheorem {rem} [tw]{Remark}

\newcommand{\bc} {\Bbb C}
\newcommand{\bn}{\Bbb N}

\newcommand{\bz}{\Bbb Z}

\newcommand{\bzk}{\bz_k}
\newcommand{\Zk}{\bz_k}

\newcommand{\alg} {\mathsf{A}}
\newcommand{\blg} {\mathsf{B}}

\newcommand{\Clg} {\mathcal{C}}
\newcommand{\Com} { \Delta}
\newcommand{\duCom} {\hat{\Com}}
\newcommand{\Cou} { \epsilon}

\newcommand{\pf} {p_{\phi}}

\newcommand {\Lin} {{\textup{Lin}}}
\newcommand {\Ker} {{\textrm{Ker}}}

\newcommand{\Alg}{\mathcal{A}}
\newcommand{\Blg}{\mathcal{B}}
\newcommand{\dual}{\hat{\Alg}}
\newcommand{\duAlg}{\hat{\Alg}}

\newcommand{\dij}{d_{i,j}}
\newcommand{\eij}{e_{i,j}}

\newcommand{\ot}{\otimes}
\newcommand{\spot}{\ot^{\textup{sp}}}

\newcommand{\ida}{\textup{id}_{\Alg}}

\newcommand{\wt}{\widetilde}

\newcommand{\id}{\textup{id}}

\numberwithin{equation}{section}

\keywords{Quantum group, idempotent state, quantum hypergroup} \subjclass[2000]{ Primary 16W30, Secondary 60B15}

\begin{document}

\author{Uwe Franz}

\address{D\'epartement de math\'ematiques de Besan\c{c}on,
Universit\'e de Franche-Comt\'e, 16, route de Gray, F-25 030 Besan\c{c}on cedex, France}
\email{uwe.franz@univ-fcomte.fr}
\curraddr{Graduate School of Information Sciences, Tohoku University, Sendai 980-8579, Japan}
\thanks{U.F.\ was supported by a Marie Curie Outgoing International Fellowship of the EU (Contract Q-MALL MOIF-CT-2006-022137), an ANR Project (Number ANR-06-BLAN-0015), and a Polonium cooperation}

\author{Adam Skalski}

\footnote{\emph{Permanent address of the second named author:} Department of Mathematics,
University of \L\'{o}d\'{z}, ul. Banacha 22, 90-238 \L\'{o}d\'{z}, Poland.}
\address{Department of Mathematics and Statistics, Lancaster University, Lancaster, United Kingdom }
\email{a.skalski@lancaster.ac.uk}

\title{\bf On idempotent states on quantum groups }

\begin{abstract}
\noindent Idempotent states on a compact quantum group are shown to yield group-like projections in
the multiplier algebra of the dual discrete quantum group. This allows to deduce that every
idempotent state on a finite quantum group arises in a canonical way as the Haar state on a finite
quantum hypergroup. A natural order structure on the set of idempotent states is also studied and
some examples discussed.
\end{abstract}

\maketitle

In the classical theory of locally compact groups probability measures which are
idempotent with respect to the convolution play a very distinguished
role. Thanks to a classical theorem by Kawada and It\^o
(\cite[theorem 3]{kawada+ito}, see also \cite{Heyer} and references therein) we know they all arise as Haar
states on compact subgroups. An analogous statement for quantum groups has been known to be false
since 1996 when A.\,Pal showed the existence of two idempotent states $\phi_1, \phi_2$ on the 8-dimensional
Kac-Paljutkin quantum group
whose null-spaces are not selfadjoint, and therefore neither $\phi_1$ nor $\phi_2$ can
arise as the Haar state on a quantum subgroup.
Even simpler counterexamples of similar
nature can be easily exhibited on group algebras of finite noncommutative groups (see Section
\ref{Exam}).

In this paper we begin a general study of idempotent states on compact quantum groups, i.e\ those
states on compact quantum groups which satisfy the formula \[ \phi = (\phi \ot \phi)\Com,\] where
$\Com$ denotes the comultiplication. Our initial interest in such objects was related to the fact
that they naturally occur as the limits of C\'esaro averages for convolution semigroups of states
(\cite{ours}). It is not difficult to see that the non-selfadjointness of the null space of a given
idempotent state is the only obstacle for it to arise as the Haar state on a quantum subgroup.
Further recent work by A.\,Van Daele and his collaborators ([L-VD$_{1-2}$], \cite{qalghyp})
together with a basic analysis of the case of group algebras of discrete groups suggest that the
appropriate generalisation of Kawada and It\^o's theorem to the realm of quantum groups should read
as follows: all idempotent states on (locally compact) quantum groups arise in a canonical way as
Haar states on compact quantum subhypergroups. At the moment such a general result seems to be out
of our reach -- although a notion of a compact quantum hypergroup was proposed in \cite{ChaV}, it
seems to be rather technical and difficult to apply for our purposes. Nevertheless, using the
concepts of group-like projections and algebraic quantum hypergroups introduced in the earlier
mentioned papers of A.\,Van Daele, we are able to show the following: \emph{every idempotent state
on a finite quantum group $\Alg$ arises in a canonical way as the Haar state on a finite quantum
subhypergroup of $\Alg$}.

The plan of the paper is as follows: in Section \ref{gendef} we carefully explain all the terminology used above, beginning the
discussion in the wide category of algebraic quantum groups (\cite{algdual}). Section \ref{grlike} recalls the definition of a
group-like projection introduced in \cite{sub0}, and extends it by allowing the projection to belong to the multiplier algebra of
a given algebraic quantum group. It is also shown that one of the constructions of algebraic quantum hypergroups associated to a
group-like projection from \cite{qalghyp} remains valid in this wider context. Section \ref{idemp} shows that every idempotent
state on a compact quantum group $\alg$ can be viewed as a group-like projection in the multiplier of the (algebraic) dual of the
dense Hopf $^*$-subalgebra $\Alg$ and thus gives rise to a certain algebraic quantum hypergroup of a discrete type. In Section
\ref{finite} we focus on finite quantum groups and show the main result of the paper: every idempotent state on a finite quantum
group $\Alg$ arises in a canonical way as the Haar state on a finite quantum subhypergroup of $\Alg$. We also discuss briefly
when such a state is the Haar state on a quantum subgroup. Section \ref{Order} introduces the natural order on the set of
idempotent states of a given finite quantum group (analogous to the partial order on group-like projections considered in
\cite{sub1}) and shows that it makes the set of idempotents a (non-distributive) lattice. Finally Section \ref{Exam} describes
exactly the idempotent states and corresponding quantum sub(-hyper)groups for commutative and cocommutative finite quantum
groups. It also presents a family of examples on genuinely quantum (i.e.\ noncommutative and noncocommutative) finite quantum
groups of Y.\, Sekine (\cite{Sekine}).

In the forthcoming work \cite{FST} several results of this paper are
generalised to arbitrary compact quantum groups. It is also shown that for
$q\in \mathbb{R}\backslash\{-1\}$ all idempotent states on the compact quantum
groups $U_q(2)$, $SU_q(2)$, and $SO_{q}(3)$ arise as Haar states of quantum
subgroups. But for $q=-1$ the situation is different; we showed that there do
exist idempotent states on $U_{-1}(2)$ and $SU_{-1}(2)$ that do not come from quantum subgroups.

A reader interested only in the case of finite quantum groups can skip most of the discussion in
first three sections and focus on Sections \ref{finite}, \ref{Order} and \ref{Exam}, referring back
to definitions and statements when and if necessary. The symbol $\ot$ will always signify the
purely algebraic tensor product of $^*$-algebras. We will use $\Alg$ or $\Blg$ to denote purely
algebraic (often finite-dimensional) algebras and reserve $\alg$ or $\blg$ for $C^*$-algebras.

\section{General definitions}
\label{gendef}
 Although the main results and most of the examples in the paper will be related
specifically to finite quantum groups, we would like to begin the discussion in a much wider
category of algebraic quantum groups introduced and investigated by A.\,Van Daele and his
collaborators. We will freely use the language of multiplier algebras associated to nondegenerate
$^*$-algebras (see \cite{multiplier}).

\subsection*{Algebraic quantum groups and hypergroups}

Let $\Alg$ denote a nondegenerate *-algebra. Its vector space dual will be denoted by $\Alg'$, with
$\alg^*$ reserved for the space of bounded linear functionals on a $C^*$-algebra $\alg$.

\begin{deft}\label{com}
By a comultiplication on $\Alg$ is understood a linear $^*$-preseving map $\Com:\Alg \to M(\Alg \ot
\Alg)$ such that
\begin{rlist}
\item $\forall_{a, b \in \Alg} \; \Com(a) (1 \ot b) \subset \Alg \ot \Alg, (a\ot 1) \Com(b) \in \Alg \ot \Alg;$
\item  $\forall_{a, b,c \in \Alg} \; (a \ot 1 \ot 1) (\Com \ot \iota) (\Com(b)(1\ot c)) = (\iota \ot \Com)((a \ot 1)\Com(b))(1
\ot 1 \ot c)$.
\end{rlist}
\end{deft}

Given a pair $(\Alg, \Com)$ as above we can for any $\phi \in \Alg'$ define maps $L_{\phi}:\Alg \to M(\Alg)$, $R_{\phi}:\Alg \to
M(\Alg)$ by the formulas ($a,b \in \Alg$)
\[(L_{\phi} (a)) (b) = (\phi \ot \iota) ((\Com(a) (1\ot b)),\]
\[(R_{\phi} (a)) (b) = (\iota \ot \phi) ((\Com(a) (b\ot 1)).\]
Note that in the second formula we use the fact that by the *-property also elements of the type $\Com(a) (b\ot 1)$ sit in $\Alg
\ot \Alg$.

\begin{deft}
Let $(\Alg, \Com)$ be as in the Definition \ref{com}. A functional $\Cou \in \Alg'$ is called a counit if it is
multiplicative, selfadjoint  and for all $a\in \Alg$
\[L_{\Cou}(a) = a, \;\;R_{\Cou}(a) = a.\]
A functional $h \in \Alg'$ is called left-invariant if for all $a\in \Alg$
\[L_{h}(a) = h(a)1.\]
It is called right-invariant if for all $a\in \Alg$
\[R_{h}(a) = h(a)1.\]
\end{deft}

There is a natural notion of faithfulness for functionals on $\Alg$: a functional $\psi \in \Alg'$ is called faithful if given $a\in \Alg$
the condition $\psi(ab)=0$ for all $b \in \Alg$ implies that $a=0$.

\begin{deft}
Let $(\Alg, \Com)$ be as in the Definition \ref{com} and assume that $h \in \Alg'$ is a left-invariant faithful
functional. If there exists a linear antihomomorphic bijection $S:\Alg \to \Alg$ such that for all
$a, b \in \Alg$
\[ S((\iota \ot h)(\Com(a)(1 \ot b))) = (\iota \ot h) ((1 \ot a ) \Com(b)),\]
then $S$ is unique and is called the antipode (relative to $h$).
\end{deft}

If $h$ above is selfadjoint, then $S(S(a)^*)^*=a$ for all $a \in \Alg$.

The following definition was introduced in \cite{qalghyp}.

\begin{deft}
A nondegenerate *-algebra with a comultiplication $\Com$, a counit $\Cou$, a faithful left-invariant functional $h$ and an
antipode $S$ relative to $h$ is called a *-algebraic quantum hypergroup.
\end{deft}

For more properties of the objects defined above, in particular for the duality theory, we refer to
\cite{qalghyp}. By Lemma 2.2 of that paper the functional $h \circ S$ is right-invariant and
faithful.

\begin{deft}
An algebra $\Alg$ equipped with a comultiplication $\Com$ is called a multiplier Hopf *-algebra if
$\Com$ is a nondegenerate $^*$-homomorphism and the maps
\[ a \ot b \longrightarrow \Com(a) (1 \ot b), \;\; a \ot b \longrightarrow (a \ot 1) \Com(b)\]
extend linearly to bijections of $\Alg \ot \Alg$.
\end{deft}

Note that when $\Alg$ is a multiplier Hopf $^*$-algebra then the comultiplication extends to a
unital $^*$-homomorphism from $M(\Alg)$ to $M(\Alg\ot \Alg)$. The second condition in Definition
\ref{com} reduces then to the usual coassociativity of the comultiplication.

\begin{deft}
A multiplier Hopf $^*$-algebra for which there exists a faithful positive left-invariant functional
$h$ is called an algebraic quantum group. It is called unimodular if $h$ is also right-invariant.
\end{deft}

Any algebraic quantum group is a $^*$-algebraic quantum hypergroup (so in particular has a unique
counit and a unique antipode relative to the fixed left-invariant functional). The
comultiplication, the counit and the antipode have respective homomorphic, homomorphic and
anti-homomorphic extensions to maps $M(\Alg)\to \bc$, $M(\Alg) \to M(\Alg \ot \Alg)$ and $M(\Alg)
\to M(\Alg)$. The extensions satisfy the same algebraic properties as the original maps - the last
fact is well-known and easy (if somewhat tedious) to establish.

\begin{deft} \label{coint}
Let $\Alg$ be an algebraic quantum group or $^*$-algebraic quantum hypergroup. It is said to be
of a compact type if $\Alg$ is unital. It is said to be of a discrete type if it has a left
co-integral, i.e.\ a non-zero element $k \in \Alg$ such that  $ak = \Cou(a) k$ for all $a \in
\Alg$.
\end{deft}

For quantum (hyper)groups of compact type the invariance conditions simplify; in case the invariant functional is positive and
normalised it is unique. In such a case we will call it the Haar state.

\begin{deft}
A state (positive normalised functional) on an algebraic quantum group or hypergroup $\Alg$ of a compact type will be called
the Haar state if
\[ (h \ot \id_{\Alg}) \Com = (\id_{\Alg} \ot h) \Com = h(\cdot) 1.\]
It is easy to see that the Haar state is both left- and right-invariant in the sense of the  definitions above.
\end{deft}

The crucial fact for us is that both the `coefficient' algebra of a compact quantum group and its discrete `algebraic quantum
group' dual fall into the category of algebraic quantum groups. In particular finite quantum groups described below are algebraic
quantum groups.

\subsection*{Compact quantum groups and compact quantum hypergroups}

The notion of compact quantum groups has been introduced in \cite{wor1}. Here we adopt the
definition from \cite{wor2} (the symbol $\spot$ denotes the spatial tensor product of
$C^*$-algebras):

\begin{deft}
A \emph{compact quantum group} is a pair $(\alg, \Com)$, where $\alg$ is a unital $C^*$-algebra,
 $\Com:\alg \to \alg \spot \alg$ is a unital, *-homomorphic map which is
coassociative:
\[ (\Com \ot \id_{\alg}) \Com = (\id_{\alg} \ot \Com) \Com\]
 and $\alg$ satisfies the quantum cancellation properties:
\[ \overline{\Lin}((1\ot \alg)\Com(\alg) ) = \overline{\Lin}((\alg \ot 1)\Com(\alg) )
= \alg \spot \alg. \]
\end{deft}

One of the most important features of compact quantum groups is the existence of the dense $^*$-subalgebra $\Alg$ (the algebra of
matrix coefficients of irreducible unitary representations of $\alg$), which is an algebraic quantum group of a compact type (in
the sense of the previous subsection). In particular we also have the following

\begin{prop} (\cite{wor2})
Let $\alg$ be a compact quantum group. There exists a unique state $h \in \alg^*$ (called the \emph{Haar state} of $\alg$) such
that for all $a \in \alg$
\[ (h \ot \id_{\alg})\circ  \Com (a) = ( \id_{\alg} \ot h)\circ  \Com (a) = h(a) 1.\]
\end{prop}

 A definition of a compact quantum hypergroup was proposed by  L.\,Chapovsky and L.\,Vainerman in \cite{ChaV}. As it is rather  technical
 (in particular apart from the Hopf-type structure the existence of modular automorphisms is assumed), we hope that in future
 some simplifications might be achieved. For our purposes it is enough to think of a compact quantum hypergroup as a unital
 $C^*$-algebra $\alg$ with a  unital, $^*$-preserving, completely bounded and coassociative, but not
 necessarily multiplicative comultiplication $\Com:\alg \to \alg \spot \alg$, equipped with a faithful Haar state.

\subsection*{Finite quantum groups and hypergroups}

Finite quantum groups can be defined in a variety of ways. In context of the previous discussion of algebraic quantum groups we
can adopt the following definition.

\begin{deft}
A finite-dimensional algebraic quantum group is called a finite quantum group.
\end{deft}

The definition above imposes the existence of the Haar state as one of the axioms. A.\,Van Daele
showed that it can be deduced from a priori weaker set of assumptions:

\begin{tw}(\cite{Haarfinite})
A finite dimensional Hopf $^*$-algebra is a finite quantum group if and only if it has a faithful representation in the algebra
of bounded operators on a Hilbert space. Each finite quantum group is of both compact and discrete types.
\end{tw}

The proof of the following facts can also be found in \cite{Haarfinite}:

\begin{lem}
If $\Alg$ is a finite quantum group then the antipode $S$ is a $^*$-preserving map satisfying $S^2=
\id_{\Alg}$ and the Haar state $h$ is a trace (i.e.\ $h(ab) = h(ba)$ for $a,b \in \Alg$).
\end{lem}

It is also possible to characterise finite quantum groups in the spirit of the Woronowicz's
definition of compact quantum group:

\begin{lem} \label{cancel}
A unital finite-dimensional $C^*$-algebra $\Alg$ with the unital $^*$-homomorphic coproduct $\Com: \Alg \to \Alg \ot \Alg$ is a
finite quantum group if and only if it satisfies  the quantum cancellation properties
\[\Lin ((\Alg \ot 1_{\Alg} ) \Com(\Alg)) = \Lin (( 1_{\Alg}\ot \Alg )  \Com(\Alg))) = \Alg \ot \Alg \]
(recall that unitality of $\Alg$ together with condition (i) in Definition \ref{com} implies that
$\Com$ is coassociative in the usual sense, i.e.\ $(\Com \ot \id_{\Alg})\Com= (\id_{\Alg} \ot \Com)
\Com$).
\end{lem}

The last two statements assert the existence of
objects such as a Haar state in (the first case) or a Haar state, an antipode and a counit (in the second case) making the $^*$-algebra in
question a finite quantum group.

We are ready to define the second class of finite-dimensional  algebras mentioned in the
introduction, namely finite quantum hypergroups.

\begin{deft}
A finite quantum hypergroup is a finite-dimensional algebraic quantum hypergroup with a faithful
left-invariant positive functional.
\end{deft}

As every finite quantum hypergroup has a canonical $C^*$-norm coming from the faithful
$^*$-representation on the GNS space of the left-invariant functional, it is automatically unital
(thus of a compact type) and the left-invariant functional may be assumed to be a state. It is also
right-invariant. Thus a finite quantum hypergroup whose coproduct is homomorphic is actually a
finite quantum group.

\subsection*{Idempotent states on compact quantum groups and Haar states on quantum subhypergroups}

Let us begin with the following definition generalising the notion of an idempotent probability
measure on a compact group:

\begin{deft}
A state  $\phi$ on a compact quantum group $\alg$ is said to be an idempotent state if
\[ (\phi \ot \phi) \Com = \phi.\]
\end{deft}

Kawada and It\^o's classical theorem states that each idempotent probability measure arises as the Haar measure on a compact subgroup.
We need therefore to introduce the notion of a quantum subgroup.

\begin{deft}
If $\alg, \blg$ are compact quantum groups and $\pi_{\blg}:\alg \to \blg$ is a surjective unital $^*$-homomorphism such that
$\Com_{\blg} \circ \pi_{\blg} = (\pi_{\blg} \ot \pi_{\blg})\circ \Com_{\alg}$, then $\blg$ is called a quantum subgroup of
$\alg$.
\end{deft}

Note that strictly speaking the definition of a quantum subgroup involves not only an algebra $\blg$ but also a morphism
$\pi_{\blg}$ describing how $\blg$ `sits' in $\alg$.

It is easy to check that if $h_{\blg}$ is the Haar state on $\blg$ then the functional $h_{\blg} \circ \pi_{\blg}$ is an
idempotent state on $\alg$ (see Proposition \ref{idem} below). As the example of A.\,Pal (\cite{Pal}) shows, not all idempotent
states arise in this way. The next observation is very simple, but as it gives the intuition for the main results of this paper,
we formulate it as a separate proposition.

\begin{prop} \label{idem}
Let $\alg$ be a compact quantum group, let $\blg$ be a unital $C^*$-algebra equipped with a
coassociative linear map $\Com_{\blg}:\blg \to \blg \spot \blg$. If $\pi:\alg \to \blg$ is a unital
positive map such that $\Com_{\blg} \circ \pi = (\pi \ot \pi)\circ \Com_{\alg}$, and $\psi$ is an
idempotent state on $\blg$ (which means that $\psi = (\psi \ot \psi)\Com_{\blg}$), then the
functional $\psi \circ \pi$ is an idempotent state on $\alg$.
\end{prop}

Below we formalise the definition of a quantum subhypergroup of a finite quantum group.

\begin{deft}
If $\Alg$ is a finite quantum group, $\Blg$ is a finite quantum hypergroup and $\pi_{\Blg}:\Alg \to
\Blg$ is a surjective unital completely positive map such that $\Com_{\Blg} \circ \pi_{\Blg} =
(\pi_{\Blg} \ot \pi_{\Blg})\circ \Com_{\Alg}$, then $\Blg$ is called a quantum subhypergroup of
$\Alg$.
\end{deft}

The definition above does not correspond to the notion of subhypergroup in the classical context
(it is not even clear whether commutative compact quantum hypergroups as defined in \cite{ChaV}
have to arise as algebras of functions on compact hypergroups), but is instead motivated by
understanding unital completely positive maps intertwining the respective coproducts as natural
morphisms in the category of compact or finite quantum hypergroups.

\begin{deft}
An idempotent state on a finite quantum group $\Alg$ will be said to \emph{arise as the Haar state on a quantum subhypergroup of
$\Alg$} if there exists $\Blg$, a finite quantum subhypergroup of $\Alg$ (with the corresponding map $\pi_{\Blg}:\Alg \to \Blg$)
such that
\[\phi =
h_{\Blg} \circ \pi_{\Blg},\]
 where $h_{\Blg}$ denotes the Haar state on $\Blg$.
\end{deft}

The definition above is not fully satisfactory as it is easy to see that given an idempotent state
$\phi$ the choice of $\Blg$ is non-unique. In particular we can always equip $\bc$ with its unique
quantum group structure and observe that $\phi$ arises as the Haar state on $\Blg=\bc$ (with
$\pi_{\Blg}:= \phi$). We can however capture the unique `maximal' choice for $\Blg$ via the
following universal property.

\begin{deft} \label{univ}
Let $\Alg$ be a finite quantum group, $\phi$ an idempotent state on $\Alg$ and let $\Blg$ be a finite quantum subhypergroup of
$\Alg$ (with the corresponding map $\pi_{\Blg}:\Alg \to \Blg$). We say that $\phi$ \emph{arises as the Haar state on $\Blg$ in a
canonical way} if $\phi = h_{\Blg} \circ \pi_{\Blg},$ where $h_{\Blg}$ denotes the Haar state on $\Blg$, and given  $\Clg$,
another finite quantum subhypergroup of $\Alg$ (with the corresponding map $\pi_{\Clg}:\Alg \to \Clg$ and the Haar state
$h_{\Clg}$) such that $\phi = h_{\Clg} \circ \pi_{\Clg}$ there exists a unique map $\pi_{\Blg \Clg}:\Blg \to \Clg$ such that
\begin{equation} \label{intertw} \pi_{\Clg} = \pi_{\Blg \Clg} \circ \pi_{\Blg}.\end{equation}
\end{deft}

Note that if a map $\pi_{\Blg \Clg}$ satisfying the intertwining formula \eqref{intertw} exists, it
is unique, is automatically surjective, linear, unital, completely positive and intertwines the
respective coproducts:
\[ \Com_{\Clg} \circ \pi_{\Blg \Clg} = (\pi_{\Blg \Clg} \ot \pi_{\Blg \Clg}) \circ \Com_{\Blg}.\]
If $\phi$ arises as the Haar state on $\Blg$ in a canonical way, then $\Blg$ is essentially unique:

\begin{tw}
Let $\Alg$ be a finite quantum group, $\phi$ an idempotent state on $\Alg$ and let $\Blg$, $\Blg'$
be finite quantum subhypergroups of $\Alg$ (with the corresponding maps $\pi_{\Blg}:\Alg \to \Blg$,
$\pi_{\Blg'}:\Alg \to \Blg'$ and the Haar states $h_{\Blg}$, $h_{\Blg'}$). Suppose that $\phi$
arises in a canonical way as the Haar state on both $\Blg$ and $\Blg'$. Then there exists a unital
$^*$-algebra and coalgebra isomorphism $\pi_{\Blg \Blg'}: \Blg \to \Blg'$ such that
\[\pi_{\Blg'} =  \pi_{\Blg \Blg'} \circ \pi_{\Blg}. \]
\end{tw}

\begin{proof}
The universal property of both $\Blg$ and $\Blg'$ guarantees the existence of surjective completely positive maps $\pi_{\Blg
\Blg'}: \Blg \to \Blg'$ and $\pi_{\Blg' \Blg}: \Blg' \to \Blg$ such that $\pi_{\Blg'} =  \pi_{\Blg \Blg'} \circ \pi_{\Blg}$ and
$\pi_{\Blg} = \pi_{\Blg' \Blg} \circ \pi_{\Blg'}$. As $\pi_{\Blg}$ and $\pi_{\Blg'}$ are surjective, it follows that $\pi_{\Blg'
\Blg} = \pi_{\Blg \Blg'}^{-1}$. It remains to recall a well known fact that a unital completely positive map from one
$C^*$-algebra onto another with a unital completely positive inverse has to preserve multiplication (it is a consequence of the
Cauchy-Schwarz inequality for completely positive maps and the multiplicative domain arguments, see for example \cite{Paulsen}).
\end{proof}

Motivated by the above result we introduce the following definition.

\begin{deft} An idempotent state $\phi$ on a quantum group $\Alg$ \emph{is the Haar state on a finite
quantum subhypergroup $\Blg$ of $\Alg$} if it arises as the Haar state on $\Blg$ in a canonical way. \label{is}\end{deft}

It is not very difficult to see that if an idempotent state on $\Alg$ arises as the Haar state on a
quantum subgroup $\Blg$ (recall that this means in particular that $\pi_{\Blg}: \Alg \to \Blg$ is a
$^*$-homomorphism), then it automatically satisfies the universal property in Definition
\ref{univ}. It can be also deduced from Theorem \ref{main} and Lemma \ref{subgroup}.

In Section \ref{finite} we will show that every idempotent state on a finite quantum group is the
Haar state on a quantum subhypergroup in the sense of Definition \ref{is}.

\section{Group-like projections in the multiplier algebra and the construction of corresponding quantum subhypergroups}
\label{grlike}

The notion of a group-like projection in an algebraic quantum group $\Alg$ was introduced by
A.\,Van Daele and M.\,Landstad in \cite{sub0} and further investigated in \cite{sub1} and
\cite{qalghyp}. Here we extend it to the case of group-like projections in the multiplier algebra
$M(\Alg)$.

\begin{deft} \label{grpdef}
Let $\Alg$ be an algebraic quantum group. A non-zero element $p\in M(\Alg)$ is called a group-like
projection if $p=p^*, \; p^2 =p$ and \begin{equation} \label{grproj} \Com(p) (1 \ot p) = p\ot
p.\end{equation}
\end{deft}

Note that the final equality above is to be understood in $M(\Alg \ot \Alg)$. By taking adjoints and applying  (the extension of)
the counit we obtain immediately that also
\[ (1 \ot p)\Com(p)  = p\ot p,  \;\; \Cou(p)=1.\]

We were not able to show that the group-like projections in the multiplier algebra automatically have to satisfy the `right'
version of the group-like property (equivalently, are invariant under the extended antipode). In the case of group-like
projections arising from idempotent states on compact quantum groups considered in Section \ref{idemp}, the properties above can
be easily established directly. To make the formulation of the results in what follows easier, we introduce another formal
definition:

\begin{deft}
Let $\Alg$ be an algebraic quantum group. A non-zero element $p\in M(\Alg)$ is called a good
group-like projection if $p=p^*, \; p^2 =p$ and \[ \Com(p) (1 \ot p) = p\ot p = \Com(p) (p \ot
1),\;\;\; S(p) = p.\]
\end{deft}

By Proposition 1.6 of \cite{sub1} any group-like projection belonging to $\Alg$ is good. The
following theorem extends Theorem 2.7 of \cite{sub1}.

\begin{tw} \label{phyper}
Let $\Alg$ be an algebraic quantum group, $p\in M(\Alg)$ a good group-like projection. A subalgebra
$\Alg_0= p \Alg p$ equipped with the comultiplication $\Com_0$ defined by
\[ \Com_0(b) = (p \ot p)( \Com(b)) (p \ot p), \; b \in \Alg_0\]
is an algebraic quantum hypergroup. If $\Alg$ is of discrete type, so is $\Alg_0$. If $\Alg$ is of a compact type, then $\Alg_0$
is of a compact type and has a positive Haar state. In particular if $\Alg$ is a finite quantum group, then $\Alg$ is a finite
quantum hypergroup.
\end{tw}

\begin{proof}
The proof is rather elementary -- we want however to carefully describe all steps, occasionally avoiding only giving proofs for
both left and right versions of the property we want to show. It is clear that $\Alg_0$ is a $^*$-subalgebra of $\Alg$. As all
our objects are effectively subalgebras of $C^*$-algebras (by \cite{Kust}), it is clear that $\Alg_0$ is nondegenerate (one can
probably find another, direct argument; the point is that $aa^*= 0$ iff $a=0$). The map $\Com_0$ has in principle values in
$M(\Alg \ot \Alg)$. However if $a,b,c \in \Alg$ then
\begin{align*}  (pap \ot pbp)&\Com_0 (pcp) =
(pap \ot pbp) (p \ot p) \Com(p) \Com(c) \Com(p) (p \ot p) \\&= (p \ot p)(pap \ot pbp) \Com(c) (p
\ot p) = (p\ot p) z (p \ot p),\end{align*} where $z= (pap \ot pbp) \Com(c) \in \Alg \ot \Alg$.
This shows that $(pap \ot pbp)\Com_0 (pcp) \in  \Alg_0 \ot \Alg_0$. Repeating the argument with
$pap \ot pbp$ on the right we obtain that $\Com_0:\Alg_0 \to M(\Alg_0 \ot \Alg_0)$.

Let us now check that $\Com_0$ is a comultiplication in the sense of Definition \ref{com}. If $a, b
\in \Alg$ then
\begin{align*}  \Com_0(pap) & (1 \ot pbp) = (p \ot p) \Com(pap) (p \ot p) (1 \ot pbp) \\&=
(p \ot p) \Com(pap) (1 \ot pbp)(p \ot p) \in (p \ot p)(\Alg \ot \Alg) (p \ot p) = \Alg_0 \ot
\Alg_0.\end{align*} Similarly $(pap \ot 1) \Com_0(pbp) \in \Alg_0 \ot \Alg_0$ and the condition
(i) is satisfied. To establish (ii) choose $a,b,c \in \Alg$ and start computing:
\begin{align*} (pap
&\ot 1  \ot 1) (\Com_0 \ot \iota) (\Com_0(pbp)(1\ot pcp)) \\&= (pap \ot 1 \ot 1)(p \ot p \ot
1)(\Com \ot \iota)((p \ot p)\Com(pbp)(p \ot p)(1 \ot pcp))  (p \ot p \ot 1).\end{align*}
 As $\Com \ot
\iota$ is a homomorphism, the latter is equal to
\begin{align*}  (pap \ot p \ot 1)&(\Com(p) \ot p) (\Com \ot \iota)(\Com(pbp)) (\Com(p) \ot pcp) (p \ot p \ot 1) \\&=
(pap \ot p \ot p)(\Com \ot \iota)(\Com(pbp)) (p \ot p \ot pcp).\end{align*}
 On the other hand, in an analogous manner,
\begin{align*}  (\iota \ot & \Com_0)((pap \ot 1)\Com_0(pbp))(1 \ot 1 \ot pcp) \\&= (1 \ot p \ot p) (\iota \ot \Com)
((pap \ot 1) (p \ot p) \Com(pbp) (p \ot p)) (1\ot p \ot p) (1 \ot 1 \ot pcp) \\&= (1 \ot p \ot p)
(pap \ot \Com(p))(\iota \ot \Com) (\Com(pbp))  (p \ot \Com(p)) (1\ot p \ot pcp)\\&=  (pap \ot p \ot
p) (\iota \ot \Com) (\Com(pbp))   (p\ot p \ot pcp).\end{align*} As $\Com$ is coassociative in the
usual sense, (ii) follows from the comparison of the formulas above.

Note that $\Com_0$ is by definition a positive map; it is even completely positive (in the obvious
sense).

Let $\Cou$ and $S$ denote respectively the counit and the antipode of $\Alg$ and  write $\Cou_0:=
\Cou|_{\Alg_0}$, $S_0 = S|_{\Alg_0}$. Then $\Cou_0$ is a selfadjoint multiplicative functional and
for all $a,b \in \Alg$
\begin{align*} (\Cou_0 \ot \iota)&(\Com_0(pap)(1 \ot pbp)) = (\Cou \ot \iota)
((p \ot p) \Com(pap) (1 \ot pbp)(p \ot p)) \\&= (\Cou \ot \iota)(p \ot p) (\Cou \ot \iota) (
\Com(pap) (1 \ot pbp)) (\Cou \ot \iota)(p \ot p) = p pap pbp p = pap pbp.\end{align*} Similarly we
can show all the remaining equalities required to deduce that $\Cou_0$ satisfies the counit
property for $(\Alg_0, \Com_0)$. Further let $h\in \Alg'$ denote a left-invariant functional on
$\Alg$ and put $h_0= h|_{\Alg_0}$. Then for any $a, b \in \Alg$
\begin{align*} (h_0 \ot \iota)&(\Com_0(pap)(1 \ot pbp)) = (h \ot \iota)
((p \ot p) \Com(pap) (p \ot p) (1 \ot pbp)) \\&= p (h \ot \iota)((p \ot 1) \Com(p) \Com(a)\Com(p)
(p \ot 1)(1\ot pbp) ).\end{align*} As taking adjoints in the defining relation for good group-like
projections yields \begin{equation} \label{switch} (p \ot 1) \Com(p) = p \ot p = (1 \ot p)
\Com(p),\end{equation} we have
\begin{align*} (h_0 \ot \iota)(\Com_0(pap)&(1 \ot pbp)) =  p (h \ot \iota)((1 \ot p) \Com(p) \Com(a) \Com(p)  (1
\ot pbp) ) p \\&= p(h \ot \iota)( \Com(pap)  (1 \ot pbp) ) = p h(pap) pbp = h_0(pap)
pbp.\end{align*} In an analogous way we can establish that a right-invariant functional on $\Alg$
yields by a restriction a right-invariant functional on $\Alg_0$ (so in particular if $\Alg$ has a
two-sided invariant functional, so has $\Alg_0$). Note also that if $h$ was faithful, so will be
$h_0$ (again one can see it via looking at the $C^*$-completions - positivity of $h$ is here
crucial). A warning is in place here - contrary to the situation in \cite{sub1} we cannot expect
here in general the invariance of $p$ under the modular group, so also if $h$ is not
right-invariant we cannot expect $h_0$ to be right-invariant.

The map $S_0$ takes values in $\Alg_0$; indeed, as $S$ (or rather its extension to $M(\Alg))$ is
anti-homomorphic, for any $a \in \Alg$
\[ S(pap) = S(p) S(pap) S(p) \in p\Alg p = \Alg_0.\]
Further if $a,b \in \Alg$
\begin{align*}
 S_0((\iota \ot h_0)&(\Com_0(pap)(1 \ot pbp))) = S ((\iota \ot h)((p \ot p)\Com(pap) (p \ot
 pbp)) \\&= S( p (\iota \ot h) (\Com(pap) 1 \ot pbp)) p) = S(p) S(\iota \ot h) (\Com(pap) 1 \ot pbp))
 S(p) \\&= p
(\iota \ot h) ((1 \ot pap ) \Com(pbp))p = (\iota \ot h)( (p\ot pap) \Com(pbp) (1\ot p))\\&= (\iota
\ot h)( (1\ot pap) \Com_0(pbp) ).\end{align*} In the second equality we used once again property
\eqref{switch}.

If $\Alg$ is of a discrete type and $k\in \Alg$ is a left co-integral, then we have $p kp = \Cou(p) kp =kp$. This implies that
$pkp$ is a left co-integral in $\Alg_0$. Indeed, for all $a \in \Alg$
\[pap pkp = pap kp = \Cou(pap) kp = \Cou(a) kp = \Cou(a) pkp. \]

If $\Alg$ is of a compact type, then $p \in \Alg$ is the unit of $\Alg_0$.
 If $h$ is the Haar state on $\Alg$,
%applying the identity $(h \ot \id_{\Alg})\Com = h(\cdot) 1$ to the equality
%\[(1 \ot p) \Com (p) \Com(a) \Com(p)(1 \ot p) = (p \ot p) \Com(a) (p \ot p), \;\;\;\; (a \in\Alg),\]
%yields
%\[ p h(pap) p = (h\ot \id_{\Alg})((p \ot p) \Com(a) (p \ot p)).\]
as $p\neq 0$ we have $h(p) >0$ and define $h_0=\frac{1}{h(p)} h|_{\Alg_0}$ is the (faithful) Haar state on $\Alg_0$ (this follows
from the arguments above but can be also checked directly).
%a positive normalised functional on $\Alg_0$. Moreover we have for all $a \in \Alg$
%\begin{align*} h(p)&(h_0\ot \id_{\Alg_0}) (\Com_0(pap)) = (h\ot \id_{\Alg}) ((p\ot p) \Com(pap) (p \ot p))  \\
%   &=(h\ot \id_{\Alg}) ((p\ot p) \Com(a) (p \ot p)) = p h(pap) p = h(p) h_0(pap) 1_{\Alg_0},\end{align*}
%so $h_0$ is left invariant. The right invariance follows in an analogous way.

The last statement follows now directly from the definitions.
%If $\Alg$ is a finite quantum group then faithfulness of $h_0$ is equivalent to the condition that $h(pap pa^*p)=0$ implies
%$pap=0$ for all $a \in \Alg$. This is again clear when we recall that $h$ is a faithful positive functional on a $C^*$-algebra
%$\Alg$.
\end{proof}

The following fact extends equivalence (i)$\Longrightarrow$(ii) in Proposition 1.10  and a part of Theorem 2.2 of \cite{sub1},
with the same proofs remaining valid.

\begin{lem} \label{centgr}
Let $p\in M(\Alg)$ a group-like projection. Then $p$ is in the center of $M(\Alg)$ if and only if
$p\Alg = \Alg p$. If this is the case and $p$ is a good group-like projection, then the
construction from Theorem \ref{phyper} yields an algebraic quantum group.
\end{lem}

\section{Idempotent states on compact quantum groups}
\label{idemp}

Let now $\alg$ be a compact quantum group, let $\Alg$ denote the Hopf $^*$-algebra of the
coefficients of all irreducible unitary corepresentations of $\alg$, let $h$ denote the Haar state
on $\alg$. Recall that $\Alg$ is an algebraic quantum group of compact type. Let  $\dual=\{_ah: a
\in \Alg\}$ denote the dual of $\Alg$ in the algebraic quantum group category ($_ah \in \alg^*$,
$_ah(b) := h(ba)$). Its coproduct will be denoted by $\duCom$. Note that (for example by
Proposition 3.11 of \cite{algdual}) $\dual=\{h_a: a \in \Alg\}$, where $h_a \in \alg^*$, $h_a(b):=
h(ab)$.

Fix also once and for all an idempotent state $\phi \in \alg^*$.

The first observation is that $\phi$ is invariant under the antipode, in the sense that
\begin{equation}\phi(S(a)) = \phi(a), \;\;\; a \in \Alg.\label{invant}\end{equation} Probably the easiest way to see
it is to observe that if $U \in M_n(\alg)$, $U = \sum_{i,j=1}^n e_{ij} \ot a_{ij}$ is an
irreducible corepresentation of $\alg$, then  the matrix $(\iota \ot \phi) (U) =
(\phi(a_{ij}))_{i,j=1}^n$ is an idempotent contraction. This implies that it must be selfadjoint,
so that $\phi(S(a_{ij})) = \phi(a_{ji}^*)= \phi (a_{ij})$.

Further note that $\phi$ yields in a natural way a multiplier of $\hat{\Alg}$. Indeed, for $a \in \Alg$
\begin{align*}(\phi \ot
\,_bh) \Com (a) &=(\phi \ot h) (\Com(a)(1 \ot b)) = (\phi\circ S \ot h)((1 \ot a) \Com(b))\\&=
(\phi\ot h)((1 \ot a) \Com(b)) =h (a L_{\phi}(b)) =\, _{L_{\phi}(b)}h(a).\end{align*} In the same
way we obtain the formula
\begin{align*}(_b h\ot \phi ) \Com (a) =\,_{R_{\phi}(b)}h(a).\end{align*}
The fact that $\phi$ yields a multiplier follows now from the associativity of the convolution. It
will be denoted further by $p_{\phi}$. The formulas above, together with the analogous formulas for
the functionals of the type $h_a$ give then:
\[ \pf \,_bh =\, _{L_{\phi}(b)}h, \;\;\; _bh \,\pf =\, _{R_{\phi}(b)}h, \;\; \pf\, h_b = h_{L_{\phi}(b)}, \;\;\; h_b \, \pf = h_{R_{\phi}(b)}.\]

\begin{lem} \label{grouplike}
The element $\pf$ defined above is a good group-like projection in $M(\dual)$.
\end{lem}
\begin{proof}
Intuitively the claim is obvious, let us however provide a careful argument. For any $b \in \Alg$
\[ (\pf \pf) (h_b) = \pf (\pf (h_b)) = \pf (h_{L_{\phi}(b)})= h_{L_{\phi}(L_{\phi}(b))}= h_{L_{\phi}(b)} = \pf (h_b).\]
Further recall that $(_bh)^* (a) = \overline{_bh(S(a)^*)}$, so that $(_bh)^* = _{S(b)^*}h$.
Therefore
\[ (\pf) ^* \,_bh = ((_bh)^*\,  \pf)^* = (_{S(b)^*}h \,\pf)^* = (_{R_{\phi}(S(b)^*)}h)^* =
 _{S((R_{\phi}(S(b)^*))^*}h\]
Note now that as $\phi$ is selfadjoint, $R_{\phi} (a^*) = (R_{\phi}(a))^*$ for all $a \in \Alg$;
moreover as $\phi$ is $S$-invariant,
\[ R_{\phi}(S(a)) = (\iota \ot \phi)\circ \Com \circ S (a) = (\iota \ot \phi) \circ \tau \circ (S \ot S)\Com(a) =
(\phi \circ S \ot S)\Com(a) = S(L_{\phi}(a)).\]
  This implies that
\[ S(R_{\phi}(S(b)^*)) = S(((R_{\phi} (S(b))^*) = S((S(L_{\phi}(b)))^*).\]
Finally $(S(R_{\phi}(S(b)^*))^*= L_{\phi}(b)$ and $\pf^* =\pf$. It remains to establish the
group-like property.  As the multipliers on both sides are clearly selfadjoint, it is enough to
show that
\[z \duCom(\pf) (1 \ot \pf) = z(\pf \ot \pf)\]
for all $z\in \dual \ot \dual$.  Using the `quantum cancellation properties' it is equivalent to
establishing that for all $b, a \in \Alg$
\begin{equation} \label{grlk} ( _b h \ot 1)
\duCom(_a h) \duCom(\pf) (1 \ot \pf) = (_b h \ot 1) \duCom(_a h) (\pf \ot \pf).\end{equation} By the argument contained in the
proof of Lemma 4.5 of \cite{algdual}, if $b \ot a = \sum_{i=1}^n(1 \ot c_i) \Com(d_i)$ for certain $n \in \bn$, $c_1, \ldots,
c_n, d_1, \ldots, d_n \in \Alg$, then
\[ (_b h \ot 1) \duCom(_a h) = \sum_{i=1}^n  {_{d_i}}h \ot\,  _{c_i}h.\]
It remains to observe that if $b \ot a$ decomposes as above, then by coassociativity we obtain $b
\ot R_{\phi}a = \sum_{i=1}^n(1 \ot c_i) \Com(R_{\phi} d_i)$. Therefore the left hand side of
\eqref{grlk} is equal to \begin{align*} ( _bh \ot 1) \duCom(_a h \,\pf) (1 \ot \pf) &= (_bh \ot 1)
\duCom(_{R_{\phi}(a)}h) (1 \ot \pf) = (\sum_{i=1}^n {_{R_{\phi}d_i}}h \ot\, _{c_i}h) (1 \ot \pf)
\\&= \sum_{i=1}^n {_{R_{\phi}d_i}}h \ot\, _{R_{\phi}c_i}h,\end{align*} whereas the right hand side
equals
\[ (_bh \ot 1) \duCom(_ah) (\pf \ot \pf) = (\sum_{i=1}^n {_{d_i}}h \ot\, _{c_i}h) (\pf \ot \pf) =
\sum_{i=1}^n {_{R_{\phi}d_i}} h\ot\, _{R_{\phi}c_i}h.\]

As stated in the comments after Definition \ref{grpdef}, to conclude the argument it is enough to
establish that $\hat{S}(\pf) = \pf$.  Recall that the antipode $\hat{S}$ in $\dual$ is defined by
\[\hat{S} (\omega) = \omega \circ S, \;\;\;   \omega \in \dual\]
($S$ denotes the antipode of $\Alg$). This together with the antihomomorphic property of $S$ implies that
\[ \hat{S} (_a h) = h_{S(a)}\]
and further
\[ \pf \hat{S}(_a h) = \pf h_{S(a)} = h_{L_{\phi}(S(a))}, \;\;\; \hat{S}(_a h \pf) = \hat{S} (_{R_{\phi}(a)} h) =
h_{S(R_{\phi}(a))}\] ($a \in \Alg$). It remains to observe that
\[ S(R_{\phi}(a)) = (S \ot \phi) (\Com(a)) = (\iota \ot \phi)(S\ot S) (\Com(a)) = (\phi \ot \iota) \Com(S(a)) = L_{\phi}(S(a)).\]
The equality $ \hat{S}(_a h) \pf= \hat{S}(\pf {_a} h)$ is obtained in the identical way.
\end{proof}

The above lemma in conjunction with the Theorem \ref{phyper} yields the following result.

\begin{cor}
Let $\alg$ be a compact quantum group and $\phi$ be an idempotent state on $\alg$. Let $p_{\phi}$ be a multiplier of $\duAlg$
associated with $\phi$. The algebra $\duAlg_{\phi} := p_{\phi} \duAlg p_{\phi}$, equipped with the natural coproduct, counit,
antipode and left-invariant functional is an algebraic quantum hypergroup of a discrete type.
\end{cor}

We stated in the introduction that we would like to show that any idempotent state on a quantum group is the Haar state on a
quantum subhypergroup. The problem with the construction above lies in the fact that it only provides a quantum subhypergroup of
$\duAlg$ and its dual is not the hypergroup we are looking for. In the case when $p_{\phi}$ actually lies in the algebra
$\duAlg$ we can make use of the Fourier transform of $p_{\phi}$ and thus pull the construction back to $\alg$. This will be done
in the next section in the context of finite quantum groups.

\section{Idempotent states on finite quantum groups are Haar states on quantum subhypergroups}
\label{finite}

In this section we show that every idempotent state on a finite quantum group $\Alg$ is the Haar state on a finite quantum
subhypergroup of $\Alg$.

We start with the following observation.

\begin{lem} \label{corr}
Let $\Alg$ be a finite quantum group. There is a one-to-one correspondence between idempotent
states on $\Alg$ and group-like projections in $\duAlg$.
\end{lem}

\begin{proof}
Let $\phi\in \Alg'$ be an idempotent state. Lemma \ref{grouplike} shows immediately that $\phi$
viewed as an element of $M(\duAlg)= \duAlg$ is a (good) group-like projection.

Conversely, suppose that $p\in \duAlg$ is a group-like projection. Then $p$ corresponds (via the
vector space identification) to a functional $\psi$ in $\Alg'$. The functional $\psi$ is a non-zero
idempotent (as the multiplication in $\duAlg$ corresponds to the convolution on $\Alg^*$). It is
thus enough if we show it is positive. As the Fourier transform (see \cite{sub1}) is a surjection
from $\Alg$ to $\duAlg$,  there exists a unique element $\hat{p} \in \Alg$ such that $\psi =\,
_{\hat{p}}h$. Proposition 1.8 of \cite{sub1} implies that $\hat{p}$ is a positive scalar multiple
of a group-like projection -- the scalar is related to the proper normalisation of the Fourier
transform. Using the tracial property of $h$ we obtain that
\[ \psi(a) = h(\hat{p} a \hat{p}), \;\;\; a \in \Alg,\]
and positivity of $\psi$ follows from the positivity of $h$.
\end{proof}

The lemma above can be rephrased in the following form, which will be of use in Theorem \ref{main}.

\begin{cor} \label{corol}
Let $\Alg$ be a finite quantum group and let $\phi\in \Alg'$. The following are equivalent:
\begin{rlist}
\item $\phi$ is an idempotent state;
\item there exists a group-like projection $p \in \Alg$ such that
\begin{equation}\phi(a) = \frac{1}{h(p)} h(p a p), \;\;\; a \in \Alg.\label{pphi}\end{equation}
\end{rlist}
\end{cor}

\begin{proof}
The implication (i)$\Longrightarrow$(ii) was established in the proof of Lemma \ref{corr}. The implication
(ii)$\Longrightarrow$(i) uses once again tracial property of $h$, Proposition 1.8 of \cite{sub1} and the correspondence in Lemma
\ref{corr}.
\end{proof}

 In \cite{Haarfinite} A.\,Van Daele showed that every finite quantum group $\Alg$ possesses a (unique) element  $\eta \in \Alg$ such that
\[\Cou(\eta) =1, \;\;\; a \eta = \Cou(a) \eta, \;\; a \in \Alg.\]
It is called the \emph{Haar element} of $\Alg$ (note that the first condition is simply a choice of
normalisation and the second means that $\eta$ is a co-integral in the sense of Definition
\ref{coint}). We automatically have $h(\eta) \neq 0$. It turns out that one can actually describe
the projection $\hat{p}$ corresponding to an idempotent state $\phi$ directly in terms of $\phi$
and $\eta$. The lemma below has to be compared with the more general discussion of inverse Fourier
transforms in Section \ref{Order} (see \cite{vDFourier}).

\begin{lem} \label{Haarform}
Let $\Alg$ be a finite quantum group and let $\phi\in \Alg^*$ be an idempotent state. The projection $\hat{p}_{\phi}$ associated
to $\phi$ by Corollary \ref{corol} is given by the formula
\[ \hat{p}_{\phi} = \frac{\phi(\eta)}{h(\eta)}(\phi \ot \ida)(\Com (\eta)).\]
\end{lem}
\begin{proof}
Let $r = (\phi \ot \ida)(\Com (\eta))= (\phi \circ S \ot \ida)(\Com (\eta))$ (recall that
$\phi\circ S= \phi$). Then for any $a \in \Alg$ using the Sweedler notation we obtain
\begin{align*}h(ra) &= (\phi \ot h) ((S \ot \ida) (\Com(\eta) (1 \ot a))
= (\phi \ot h) ( (1 \ot \eta) \Com(a)) \\ &= (\phi \ot h) (a_{(1)} \ot \Cou(a_{(2)}) \eta) = \phi(a) h(\eta).
\end{align*}
 This means that if  $s = \frac{1}{h(\eta)} r$, then $h(s)=1$ and
\[ \phi(a)= \frac{1}{h(s)} h(sa).\]
Comparison with the formula \eqref{pphi} shows that $\hat{p}_{\phi}$ has to be a scalar multiple of $s$ (as the Haar functional
is here a faithful trace). As we know that $\Cou(\hat{p}_{\phi})=1$, the correct normalisation is given by $\hat{p}_{\phi} =
\phi(\eta)s$. Note that this in particular implies that we must have $\phi(\eta) > 0$ (positivity of $\eta$ is established in
\cite{Haarfinite}).
\end{proof}

Corollary \ref{corol} together with Lemma \ref{phyper} yields the following
result providing an appropriate generalisation of Kawada and It\^o's classical theorem to the category of finite quantum groups.

\begin{tw} \label{main}
Let $\Alg$ be a finite quantum group and let $\phi \in \Alg'$ be an idempotent state. Then
%there exists $\Alg_{\phi}$, a quantum subhypergroup of $\Alg$, such that $\Alg_{\phi}$
$\phi$ is the Haar state on a quantum subhypergroup of $\Alg$.
\end{tw}
\begin{proof}
Let $p$ be a group-like projection in $\Alg$ such that
\[\phi(a) = \frac{1}{h(p)} h(p a p), \;\;\; a \in \Alg\]
($h$ denotes the Haar state on $\Alg$, see Corollary \ref{corol}). Put $\Alg_{\phi} = p \Alg p$ and
equip it with the finite quantum hypergroup structure discussed in Theorem \ref{phyper}. It is
immediate that the map $\pi:\Alg \to p \Alg p$ is a unital completely positive surjective map
intertwining the corresponding coproducts. As the functional $pap \longrightarrow h(pap)$ is both
left- and right- invariant with respect to the coproduct in $p \Alg p$, it is clear that the Haar
state on $\Alg_{\phi}$ is given by the formula $h_{\Blg} (pap) = \frac{1}{h(p)} h(p a p)$ and $\phi
= h_{\Blg} \circ \pi.$

It remains to show that the pair $(\Alg_{\phi},\pi)$ satisfies the universal property from
Definition \ref{univ}. Suppose then that $\Clg$ is a quantum subhypergroup of $\Alg$, with the Haar
state $h_{\Clg}$ and the corresponding unital surjection $\pi_{\Clg}: \Alg \to \Clg$, such that
$\phi = h_{\Clg} \circ \pi_{\Clg}$. Then $h_{\Clg} (\pi_{\Clg} (1-p)) =\phi(1-p) = 0$ and the
faithfulness of $h_{\Clg}$ and positivity of $1-p$ imply that $\pi_{\Clg} (p) =1$. As $\pi_{\Clg}$
is completely positive, the multiplicative domain arguments (Theorem 3.18 in \cite{Paulsen})  imply
that $\pi_{\Clg} (pap) = \pi_{\Clg} (a)$ for all $a \in \Alg$. A moment's thought shows that this
implies the existence of a map $\pi_{\Alg_{\phi} \Clg}: \Alg_{\phi} \to \Clg$ such that $\pi_{\Clg}
=  \pi_{\Alg_{\phi} \Clg} \circ \pi$.
\end{proof}

It is natural to ask when an idempotent state on $\Alg$ arises as the Haar state on a quantum
\emph{subgroup} of $\Alg$. The answer is provided by the characterisation of the null space.

\begin{tw} \label{equivHaar}
Let $\Alg$ be a finite quantum group and $\phi \in \Alg'$ and idempotent state. The following are
equivalent:
\begin{rlist}
\item $\phi$ is the Haar state on a quantum subgroup of $\Alg$;
\item the null space of $\phi$, $N_{\phi} = \{a \in \alg: \phi(a^* a) = 0\}$, is a two-sided
(equivalently, selfadjoint, equivalently, $S$-invariant) ideal of $\Alg$;
\item the projection $\hat{p}_{\phi}$ associated to $\phi$ according to Corollary \ref{corol} is in the center of $\Alg$.
\end{rlist}
\end{tw}

\begin{proof}
As by Schwarz inequality it is easy to see that $N_{\phi}$ is always a left ideal of $\Alg$, it is a two-sided ideal if and only
if it is selfadjoint. Further as $\phi$ is invariant under the antipode and the antipode on a finite quantum group is a
$^*$-preserving antihomomorphism, we have $a \in N_{\phi}$ if and only if $S(a) \in N_{\phi}^*$ and equivalences in (ii) follow.
The idempotent property of $\phi$ implies that $a \in N_{\phi}$ if and only if $\Com(a) \in N_{\phi \ot \phi} = \Alg \ot N_{\phi}
+ N_{\phi} \ot \Alg$. Further $a \in N_{\phi}$ if and only if $h(\hat{p}_{\phi} a^* a \hat{p}_{\phi}) =0$ if and only if $a
\hat{p}_{\phi}=0$ (the Haar state is faithful). Thus
 \begin{equation} \label{Nphi} N_{\phi}=\{a \in \Alg: a
\hat{p}_{\phi} =0\}.\end{equation}

Assume that (i) holds, that is there exists a (finite) quantum group $\Blg$ and a $^*$-homomorphism $\pi: \Alg \to \Blg$ such
that $\phi = h_{\Blg} \circ \pi,$ where $h_{\Blg}$  is the Haar state on $\Blg$. As Haar states on finite quantum groups are
automatically faithful, we obtain the following string of equivalences ($a \in \Alg$):
\begin{align*} a \in N_{\phi} &
\Leftrightarrow h_{\Blg} (\pi(a^* a)) = 0 \Leftrightarrow h_{\Blg} (\pi(a)^* \pi(a))= 0 \\
& \Leftrightarrow \pi(a) = 0  \Leftrightarrow \pi(a^*) = 0 \Leftrightarrow h_{\Blg} (\pi(a)
\pi(a)^*)= 0 \Leftrightarrow \phi(a a^*) = 0 \Leftrightarrow a \in N_{\phi}^*.
\end{align*}
Thus $N_{\phi}$ is selfadjoint and (ii) is proved.

Suppose now that (ii) holds. Consider the (unital) $^*$-algebra $\Blg:=\Alg/N_{\phi}$ and let $q:\Alg \to \Blg$ denote the
canonical quotient map.  As $\Blg \ot \Blg$ is naturally isomorphic to $(\Alg \ot \Alg)/(N_{\phi} \ot \Alg + \Alg \ot N_{\phi})$,
the remarks in the beginning of the proof show that the map
\[ \Com_{\Blg} ([a]) = (q \ot q) (\Com_{\Alg}(a)), \;\; a \in \Alg,\]
is a well defined coassociative $^*$-homomorphism from $\Blg$ to $\Blg \ot \Blg$. It can be checked that both the counit and the
antipode preserve $N_{\phi}$ and thus yield maps on $\Blg$ satisfying analogous algebraic properties; alternatively one can use
the characterization in Lemma \ref{cancel} and observe that the fact that $\Blg$ is a $C^*$-algebra satisfying the cancellation
properties follows immediately from the corresponding statements for $\Alg$. Therefore $\Blg$ is a finite quantum group and
$q:\Alg \to \Blg$ is the desired surjection intertwining the respective coproducts. It remains to show that $\phi = h_{\Blg}
\circ q$; in other words one has to check that the prescription $\psi ([a]) = \phi(a), a \in \Alg$ yields the bi-invariant
functional on $\Blg$. The last statement is equivalent to the following:
\[ \forall_{a \in\Alg} \;\; ((\phi \ot \id) \Com(a) - \phi(a) 1) \in N_{\phi}, \;\;
(( \id \ot\phi ) \Com(a) - \phi(a) 1) \in N_{\phi}.\] These formulas can be checked directly using
the idempotent property of $\phi$.

The implication (iii)$\Longrightarrow$(ii) follows immediately from \eqref{Nphi}. Assume then again that (ii) holds. As $\Alg$ is
a finite-dimensional $C^*$-algebra, it is a direct sum of matrix algebras and all of its selfadjoint ideals are given by a direct
sum of some matrix subalgebras of $\Alg$. Therefore $\hat{p}_{\phi}$ has to be given by a direct sum of units in the matrix
subalgebras of $\Alg$ which do not appear in $N_{\phi}$ and is therefore central.
\end{proof}

Note that the lemma above gives in particular a new proof of the known fact that a faithful idempotent state on a finite quantum
group $\Alg$ has to be the Haar state. The equivalence of conditions (i) and (ii) persists also in the case of arbitrary compact
quantum groups (see \cite{FST}).

To simplify the notation in what follows we introduce the following definition:
\begin{deft}
An idempotent state on a finite quantum group is said to be a Haar idempotent if it satisfies the equivalent
conditions in the above theorem. Otherwise it is called a non-Haar idempotent.
\end{deft}

As expected, in case the idempotent state $\phi$ is Haar, the construction in Theorem \ref{main} actually yields a quantum
subgroup (and not only a quantum subhypergroup) of $\Alg$. We formalise it in the next lemma:

\begin{lem} \label{subgroup}
Let $\phi$ be a Haar idempotent on a finite quantum group $\Alg$ and let $p$ be a group-like
projection described in Corollary \ref{corol}. Then the map
\[ \Alg/N_{\phi} \ni [a] \longrightarrow  pap  \in  p\Alg p \]
yields an isomorphism of finite quantum hypergroups $p \Alg p$ and $\Alg/N_{\phi}$. In particular,
the coproduct in $p \Alg p$ is a $^*$-homomorphism and $p \Alg p$ is a finite quantum group.
\end{lem}

\begin{proof}
Note first that the map above is well defined. This is implied by the following string of
equivalences ($ a \in \Alg$, $h$ is the tracial Haar state on $\Alg$):
\begin{align*}
pap = 0& \Leftrightarrow h(pa^*pa p )= 0 \Leftrightarrow \phi(a^*pa) = 0 \Leftrightarrow pa \in
N_{\phi}\\
&\Leftrightarrow a^*p \in N_{\phi} \Leftrightarrow \phi (paa^*p)=0 \Leftrightarrow h(paa^*p)=0
\Leftrightarrow h(p a a^*) =0\\
&\Leftrightarrow\phi(aa^*) =0 \Leftrightarrow a \in N_{\phi}^* \Leftrightarrow a \in N_{\phi}.
\end{align*}
Denote by $q$ the canonical quotient map from $\Alg \to \Alg/N_{\phi}$. Then the map defined in the
lemma can be described simply as $j(q(a)) = pap$, $a \in \Alg$. The equivalences above imply that
$j$ is a $^*$-algebra isomorphism, so that it remains to show that it preserves the quantum
hypergroup structure. This is elementary, so we will only provide an example of a calculation with
the coproduct (again $a \in \Alg$):
\begin{align*} (j \ot j)(\Com_{\Alg/N_{\phi}} (q(a))) &= (j\ot
j) ((q \ot q)(\Com_{\Alg}(a)) = (p \ot p) \Com_{\Alg} (a) (p \ot p) \\&= (p \ot p) \Com_{\Alg}(pap)
(p\ot p)= \Com_{p \Alg p} (pap) =  \Com_{p \Alg p} (j (q (a)).\end{align*}
 We used once more the
fact that $p \in \Alg$ is a group-like projection.
\end{proof}

\section{The order structure on idempotent states on a finite quantum group} \label{Order}

In this section we introduce a natural order relation on the set of idempotent states on a fixed
finite quantum group $\Alg$ and discuss its basic properties. As in this section we will use two
different products on $\Alg' \approx\duAlg$ (vector space identification), the standard
convolution-type product will be denoted by $\star$, so that for $\phi, \psi \in \Alg'$
\[ \phi\star \psi := (\phi \ot \psi)\Com.\]

\subsection*{Order relation and  supremum for idempotent states} The order relation we introduce generalises the usual inclusion
relation for subgroups of a given group.

\begin{deft}
Let $\Alg$ be a finite quantum group and let $\mathcal{I}(\Alg)\subseteq \Alg'$ denote the set of
idempotent states on $\mathcal{A}$. Denote by $\prec$ the partial order on $\mathcal{I}(\Alg)$
defined by by \[\phi_1 \prec\phi_2 \;\; \textrm{if }\;\; \phi_1\star\phi_2=\phi_2.\]
\end{deft}

In this order the Haar state is the biggest idempotent on $\Alg$, and the counit $\Cou$ is the
smallest idempotent.

\begin{lem}\label{lem-21=1}
Let $\phi_1, \phi_2$ be idempotent states on $\Alg$. Then the following are equivalent.
\begin{rlist}
\item
$\phi_1\star\phi_2=\phi_2$;
\item
$\phi_2\star\phi_1=\phi_2$.
\end{rlist}
\end{lem}
\begin{proof}
Recall that by \eqref{invant} $\phi\circ S=\phi$ for idempotent states on finite quantum groups.
Thus if (i) holds then
\begin{eqnarray*}
\phi_2 = \phi_2\circ S = (\phi_1\otimes\phi_2)\circ\Delta\circ S  = \big((\phi_2\circ S)\otimes (\phi_1\circ S)\big) \circ \Delta
= \phi_2\star\phi_1.
\end{eqnarray*}
\end{proof}

The above fact also clearly follows from the dual point of view - two projections on a Hilbert space commute if and only if their
product is a projection.

The following lemma establishes some relation between the pointwise order of idempotent states and $\prec$.
\begin{lem}
Let $\phi_1$ and $\phi_2$ be two idempotent states on a finite quantum group $\Alg$. If there
exists $\lambda>0$ such that $\phi_1\le \lambda\phi_2$, then $\phi_1\prec\phi_2$.
\end{lem}
\begin{proof}
Apply Lemma 2.2 of \cite{Haarcomp} with $\omega=\varphi=\phi_2$ and $\rho=\phi_1/\lambda$.
\end{proof}

For any functional $\phi \in \Alg'$ we write $\phi^{\star 0}:= \Cou$.

\begin{deft} \label{supr}
Let $\phi_1$ and $\phi_2$ be two idempotent states on a finite quantum group $\Alg$. We define
\[
\phi_1\vee\phi_2=\lim_{n\to\infty} \frac{1}{n}\sum_{k=0}^{n-1} (\phi_1\star\phi_2)^{\star k},
\]
it is clear by construction that $\phi_1\vee\phi_2$ is again an idempotent state.
\end{deft}

The limit above can be understood for example in the norm sense, as $\Alg'$ is finite-dimensional.
We will use the notation $C_n(\phi)=\frac{1}{n}\sum_{k=0}^{n-1} \phi^{\star
  k}$ for finite Ces\`aro averages ($ n\in \bn, \phi \in \Alg'$).

\begin{lem}\label{prop-sup}
Let $\phi_1$, $\phi_2$, and $\phi_3$ be idempotent states on a finite quantum group $\Alg$. Then
the following properties hold:
\begin{enumerate}
\item
$\phi_i\star(\phi_1\vee \phi_2) = \phi_1\vee\phi_2 = (\phi_1\vee\phi_2)\star \phi_i$, i.e.\
$\phi_i\prec(\phi_1\vee \phi_2)$ for $i=1,2$;
\item
if $\phi_1 \prec \phi_3$ and  $\phi_2 \prec \phi_3$, then $(\phi_1\vee\phi_2)\prec\phi_3$.
\end{enumerate}
\end{lem}
\begin{proof}
\begin{enumerate}
\item
$\phi_1\star(\phi_1\vee \phi_2) = \phi_1\vee\phi_2$ is clear, since $\phi_1\star C_n(\phi_1\star\phi_2)=C_n(\phi_1\star\phi_2)$
for all $n\in\mathbb{N}$. Then $\phi_1\vee\phi_2 = (\phi_1\vee\phi_2)\star \phi_1$ follows by Lemma \ref{lem-21=1}.
\item
$\phi_1 \prec \phi_3$ and  $\phi_2 \prec \phi_3$ implies $(\phi_1\star\phi_2)^{\star k}\star\phi_3=\phi_3$ for all
$k\in\mathbb{N}$, therefore $C_n(\phi_1\star\phi_2)\star\phi_3=\phi_3$ for all $n\in\mathbb{N}$ and
$(\phi_1\vee\phi_2)\star\phi_3 = \phi_3$.
\end{enumerate}
\end{proof}
This proposition shows that the operation $\vee$ gives the supremum for the order structure defined
by $\prec$.

By Lemma \ref{grouplike} an idempotent state $\phi\in \Alg'$ can be viewed as a good group-like
projection $p_\phi$ in $M(\duAlg) = \duAlg$ and therefore Theorem \ref{phyper} allows to associate
an algebraic quantum hypergroup $\duAlg_0=p_\phi\duAlg p_\phi$ to it. We call $\phi$ a {\em
central} idempotent if $p_\phi$ belongs to the center of $\duAlg$. Lemma \ref{centgr} implies that
in this case $\duAlg_0$ is
  actually an algebraic quantum group.

The following is obvious, since sums and products of central elements are again central.

\begin{prop}\label{prop-central-id}
Let $\Alg$ be a finite quantum group. If $\phi_1,\phi_2\in \mathcal{I}(\Alg)$ are central
idempotents, then $\phi_1\vee\phi_2$ is also a central idempotent.
\end{prop}

All results of this subsection have natural counterparts for idempotent states on compact quantum
groups. The limit in the Definition \ref{supr} has to be then understood in the weak$^*$ sense and
we need to exploit certain ergodic properties of iterated convolutions, as discussed in
\cite{ours}.

\subsection*{Duality and infimum}
In this subsection we exploit the fact that in the finite-dimensional framework the Fourier
transform reverses the order and allows us to define also an infimum.

Since $\Alg$ is finite-dimensional and since the Haar state $h$ is faithful, for any functional
$\phi\in \Alg'$  there exists a unique element $\mathcal{F}^{-1}\phi \in \Alg$ such that
\begin{equation}\label{eq-fourier}
\phi(a) = h\big(a(\mathcal{F}^{-1}\phi)\big)
\end{equation}
for all $a\in \Alg$. $\mathcal{F}^{-1}\phi$ is the inverse Fourier transform of $\phi$, as defined
in Definition 1.3 of \cite{vDFourier}. In the notation used earlier we have $\phi = \,
_{\mathcal{F}^{-1}\phi} h$. Since the element $\hat{p}_\phi$ associated to an idempotent state in
Corollary \ref{corol} is a group-like projection and since the Haar state is a trace, we have
\begin{equation}\label{eq-p-phi}
\phi(a) = \frac{h(\hat{p}_\phi a \hat{p}_\phi)}{h(\hat{p}_\phi)} = \frac{h(a
\hat{p}_\phi)}{h(\hat{p}_\phi)}
\end{equation}
for all $a\in \Alg$, and therefore we have the following result ($\eta$ denotes the Haar element of $\Alg$, defined before Lemma
\ref{Haarform}).

\begin{lem} \label{Fouass}
The inverse Fourier transform of an idempotent state $\phi \in \Alg'$ and its associated (according
to Corollary \ref{corol}) projection $\hat{p}_\phi$ are related by the following formulas:
\[
\mathcal{F}^{-1}\phi = \frac{1}{h(\hat{p}_\phi)} \hat{p}_\phi,
\]
\[
\hat{p}_\phi = \frac{\phi(\eta)}{h(\eta)} \mathcal{F}^{-1}\phi
\]
\end{lem}

\begin{proof}
The first equality follows by comparing \eqref{eq-fourier} and \eqref{eq-p-phi}. Taking $a=\eta$ in
\eqref{eq-p-phi} we get
\[
h(\hat{p}_\phi) = \frac{h(\eta)}{\phi(\eta)}
\]
and the second equation follows.
\end{proof}
We use this relation to extend the definition of $\hat{p}_\phi$ to arbitrary linear functional $\phi \in \Alg'$.

As in Proposition 2.2 of \cite{vDFourier} we can define a new multiplication for functionals on
$\Alg$ that is transformed to the usual product in $\Alg$ by the inverse Fourier transform. In the
following we use the Sweedler notation.

\begin{prop}
Let $\phi_1,\phi_2\in \Alg'$. Then we have
\[
\mathcal{F}^{-1}(\phi_1\circledast\phi_2) = (\mathcal{F}^{-1}\phi_1)(\mathcal{F}^{-1}\phi_2)
\]
where the multiplication $\circledast:\Alg^* \times \Alg^* \to \Alg^*$ is defined by
\[
\phi_1\circledast \phi_2:x\mapsto \frac{1}{h(\eta)}\phi_1\big(S^{-1}(\eta_{(2)})x\big)\phi_2(\eta_{(1)}).
\]
%\label{circle}\end{equation}
%for linear functionals $\phi_1,\phi_2:\mathcal{A}\to\mathbb{C}$.
\end{prop}
\begin{proof}
Assume that $\phi_1, \phi_2 \in \Alg', a,b \in \Alg$ and $\mathcal{F}^{-1}\phi_1=a$,
$\mathcal{F}^{-1}\phi_2=b$, i.e.\ $\phi_1=\, _a h$, $\phi_2=\,_b h$. We have to show that
$\phi_1\circledast\phi_2=\,_{ab}h$. Let $x\in\Alg$, then
\begin{eqnarray}
\Delta(\eta)(x\otimes 1) &=& \sum \Delta(\eta) x_{(1)}\otimes x_{(2)}S(x_{(3)})
\nonumber \\
&=& \sum \Delta(\eta) x_{(1)}\otimes x_{(2)}S(x_{(3)}) \nonumber \\
&=& \sum \Delta(\eta x_{(1)} ) (1\otimes S(x_{(2)}) \nonumber \\
&=& \Delta(\eta) (1\otimes S(x)) \label{eta-S}
\end{eqnarray}
i.e.\ $\Delta(\eta)(x\otimes 1) = \Delta(\eta) (1\otimes S(x))$ for all $x\in\mathcal{A}$, cf.\ the
proof of Lemma 1.2 in \cite{Haarcomp}. Let $a\in\Alg$, then
\begin{eqnarray*}
h(\eta)(\phi_1\circledast \phi_2)(x) &=&
\phi_1\big(S^{-1}(\eta_{(2)})x\big)\phi_2(\eta_{(1)}) \\
&=& h(\big(S^{-1}(\eta_{(2)})xa\big)h(\eta_{(1)}b) \\
&=& h(\big(\eta_{(2)}S(xa)\big)h(\eta_{(1)}b)
\end{eqnarray*}
where we used $h\circ S=h$ and the fact that the Haar state $h$ is a trace. Therefore
\begin{eqnarray*}
(\phi_1\circledast \phi_2)(x) &=& (h\otimes h)\big(\Delta(\eta) (1\otimes
S(xa))(b \otimes 1)\big) \\
&=& (h\otimes h)\big(\Delta(\eta) (xab \otimes 1)\big)
\end{eqnarray*}
where we used \eqref{eta-S}. Finally, using the invariance of the Haar state $h$, we get
\begin{eqnarray*}
(h\otimes h)\big(\Delta(\eta) (xab \otimes 1)\big) &=& (_{xab}h \star
h)(\eta) \\
&=& h(xab) h(\eta),
\end{eqnarray*}
i.e.\
\[
(\phi_1\circledast \phi_2)(x) = h(xab).
%=\mathcal{F}\big((\mathcal{F}^{-1}\phi_1)(\mathcal{F}^{-1}\phi_2)\big)
\]
\end{proof}

By Lemma \ref{Fouass} we obtain a simple formula for the multiplication of the associated
projections of idempotent states:

\begin{cor}
Let $\phi_1,\phi_2$ be idempotent states on $\mathcal{A}$. Then we have
\[
\hat{p}_{\phi_1}\hat{p}_{\phi_2} = \hat{p}_{\phi_1\circledast\phi_2}.
\]
\end{cor}
\begin{proof}
We have
\begin{eqnarray*}
\hat{p}_{\phi_1}\hat{p}_{\phi_2} &=&
\frac{\phi_1(\eta)\phi_2(\eta)}{(h(\eta))^2}(\mathcal{F}^{-1}\phi_1)(\mathcal{F}^{-1}\phi_2) \\
&=& \frac{\phi_1(\eta)\phi_2(\eta)}{(h(\eta))^2}\mathcal{F}^{-1}(\phi_1\circledast\phi_2)
\\
&=& \frac{\phi_1(\eta)\phi_2(\eta)}{(\phi_1\circledast\phi_2)(\eta)
  h(\eta)}\hat{p}_{\phi_1\circledast\phi_2}.
\end{eqnarray*}
But using $\phi(\eta) = h\left(\eta (\mathcal{F}^{-1}\phi)\right) =
h(\eta)\varepsilon(\mathcal{F}^{-1}\phi)$, we can show
\begin{eqnarray*}
(\phi_1\circledast\phi_2)(\eta) &=& h\left(\eta
  \mathcal{F}^{-1}(\phi_1\circledast\phi_2)\right) \\
&=& h(\eta)\varepsilon\left(\mathcal{F}^{-1}(\phi_1\circledast\phi_2\right)
\\
&=& h(\eta)\varepsilon\left((\mathcal{F}^{-1}\phi_1)(\mathcal{F}^{-1}\phi_2)\right)
\\
&=&
h(\eta)\varepsilon\left(\mathcal{F}^{-1}\phi_1)\right)\varepsilon\left(\mathcal{F}^{-1}\phi_2)\right)
\\
&=& \frac{\phi_1(\eta)\phi_2(\eta)}{h(\eta)},
\end{eqnarray*}
and we get the desired identity.
\end{proof}

The following lemma is a reformulation of Proposition 1.9 in \cite{sub1} in our language.

\begin{lem}
Let $\phi_1$ and $\phi_2$ be two idempotent states on a finite quantum group $A$. Then we have
$\phi_1\prec\phi_2$ if and only if $\phi_1\circledast\phi_2=\phi_1$.
\end{lem}

We are ready to define a candidate for the infimum operation on idempotent states

\begin{deft}
Let $\phi_1$ and $\phi_2$ be two idempotent states on a finite quantum group $\Alg$. We define $\phi_1\wedge\phi_2=\lim
\frac{1}{n}\sum_{k=0}^{n-1} (\phi_1\circledast\phi_2)^{\circledast k}$.
\end{deft}

\begin{prop}
Let $\phi_1$, $\phi_2$, and $\phi_3$ be idempotent states on a finite quantum group $\mathcal{A}$.
Then we have the following properties.
\begin{enumerate}
\item
$\phi_i\circledast(\phi_1\wedge \phi_2) =(\phi_1\wedge \phi_2)  = (\phi_1\wedge\phi_2)\circledast
\phi_i$, i.e.\ $(\phi_1\wedge \phi_2)\prec \phi_1$ for $i=1,2$;
\item
if $\phi_3 \prec \phi_1$ and  $\phi_3 \prec \phi_2$, then $\phi_3\prec(\phi_1\wedge\phi_2)$.
\end{enumerate}
\end{prop}
\begin{proof}
Analogous to the proof of Proposition \ref{prop-sup}.
\end{proof}

This proposition shows that the operation $\wedge$ gives the infimum for the order structure
defined by $\prec$.

The supremum and the infimum operations are connected by the following relation:
\begin{lem} \label{dualinf}
Let $\phi_1$ and $\phi_2$ be idempotent states on a finite quantum group $\Alg$. Then
\begin{eqnarray*}
\hat{p}_{\phi_1}\vee \hat{p}_{\phi_2} = \hat{p}_{\phi_1\wedge\phi_2}, \\
\hat{p}_{\phi_1}\wedge \hat{p}_{\phi_2} = \hat{p}_{\phi_1\vee\phi_2}
\end{eqnarray*}
where we use the duality to interprete $\hat{p}_{\phi_1}$ and $\hat{p}_{\phi_2}$ as idempotent
states on $\duAlg$.
\end{lem}

The results above can be summarised in the following statement:

\begin{tw}
$(\mathcal{I}(\Alg),\prec)$ is a lattice, i.e.\ a partially ordered set with unique supremum $\phi_1\vee\phi_2$ and infimum
$\phi_1\wedge\phi_2$. The identity for $\vee$ is the counit, the identity for $\wedge$ is the Haar state.
\end{tw}

In general $(\mathcal{I}(A),\prec)$ is not a distributive lattice, since $\vee$ and $\wedge$ do not
satisfy the distributivity relations even in the special case of group algebras or functions on a
group, cf.\ Remark \ref{rem-distributive}.

\begin{prop}\label{prop-normal}
If $\phi_1,\phi_2\in \mathcal{I}(\Alg)$ are Haar idempotents, then $\phi_1\wedge\phi_2$ is also a
Haar idempotent.
\end{prop}
\begin{proof}
If $\phi_1$ and $\phi_2$ are Haar idempotents, then $\hat{p}_{\phi_1}$ and $\hat{p}_{\phi_2}$ are in the center of $\mathcal{A}$
by Theorem \ref{equivHaar}. Constructing $\hat{p}_{\phi_1\wedge\phi_2}$ corresponds to taking the Ces\`aro limit
$\lim_{n\to\infty} \sum_{k=0}^{n-1}
  (\hat{p}_{\phi_1}\hat{p}_{\phi_2})^k$, which clearly leads to an element that
  is again in the center.
\end{proof}

The above proposition can be alternatively deduced by duality from Proposition \ref{prop-central-id} and Lemma \ref{dualinf}.

%\begin{tw}\label{thm-Haar-id-sup}
%Let $\Alg$ be a finite quantum group. If $\phi_1,\phi_2\in \mathcal{I}(\Alg)$ are Haar idempotents,
%then $\phi_1\vee\phi_2$ is also a Haar idempotent.
%\end{tw}

\section{Examples}
\label{Exam}

In this section we describe several examples of idempotent states and corresponding quantum
sub(hyper)groups.

\subsection*{Commutative case}
Let $\Alg$ be a commutative finite quantum group.  There exists a  finite group $G$ such that
$\Alg$ is isomorphic (as a quantum group) to the $^*$-algebra of  functions on $G$ with the natural
comultiplication: \[ \Com(f) (g,h) = f (gh), \;\; g,h \in G, f \in \Alg.\]
 Idempotent states on $\Alg$ correspond to idempotent
measures on $G$, and the latter are known (via Kawada and It\^o's theorem) to arise as Haar measures
on subgroups of $G$.

The order relation now corresponds to the familiar partial ordering of subgroups of a given group.
Indeed, let $G_1,G_2$ be two subgroups of $G$ and denote their Haar measures by $\mu_{G_1}$ and
$\mu_{G_2}$. Then it is straightforward to check that
\[
\mu_{g_1}\prec \mu_{G_2} \quad\mbox{ if and only if} \quad G_1\subseteq G_2
\]
and
\[
\mu_{G_1}\vee \mu_{G_2} = \mu_{G_1\vee G_2}
\]
where $G_1\vee G_2$ denotes the subgroup of $G$ that is generated by $G_1$ and $G_2$.

Since the (normalized) Fourier transform of the Haar measure of a subgroup $G_0$ is the indicator
function of $G_0$, $p_{\mu_{G_0}}=\chi_{G_0}$, we get
\[
\mu_{G_1}\circledast\mu_{G_2} = \mu_{G_1}\wedge \mu_{G_2} = \mu_{G_1\wedge
    G_2},
\]
where $G_1\wedge G_2$ denotes the intersection of $G_1$ and $G_2$.

Even in this simplest case one can see that $\mathcal{I}(\Alg)$ need not be a distributive lattice:

\begin{rem}\label{rem-distributive}
Let $G=S_3$, the permutation group of three elements, and consider the subgroups generated by the three transpositions,
$G_1=\{e,t_{23}\}$, $G_2=\{e,t_{13}\}$, $G_3=\{e,t_{12}\}$. Clearly the intersection of any two of them is the trivial subgroup,
\[
G_i\wedge G_j = \{e\},
\]
for $i\not=j$, and any two of them generate the whole group,
\[
G_i\vee G_j = G
\]
for $i\not=j$. Therefore
\begin{eqnarray*}
G_1\vee(G_2\wedge G_3) =& G_1 \;\; \not= \;\; G&= (G_1\vee G_2)\wedge (G_1\vee G_3), \\
G_1\wedge(G_2\vee G_3) =& G_1 \;\;\not=\;\; \{e\} &= (G_1\wedge G_2)\vee (G_1\wedge G_3).\\
\end{eqnarray*}
\end{rem}

\subsection*{Cocommutative case}

Suppose now that $\Alg$ is cocommutative, i.e.\ $\Com = \tau \Com$, where $\tau: \Alg \ot \Alg \to
\Alg \ot \Alg$ denoted the usual tensor flip. It is easy to deduce from the general theory of
duality for quantum groups that $\Alg$ is isomorphic to the group algebra $C^*(\Gamma)$, where
$\Gamma$ is a (classical) finite group.

\begin{tw} \label{cocom}
Let $\Gamma$ be a finite group and $\Alg=C^*(\Gamma)$. There is a one-to-one correspondence between
idempotent states on $\Alg$ and subgroups of $\Gamma$. An idempotent state $\phi\in \Alg'$ is a
Haar idempotent if and only if the corresponding subgroup of $\Gamma$ is normal.
\end{tw}
\begin{proof}
The dual of $\Alg$ may be identified with the usual algebra of functions on $\Gamma$. The
convolution of functionals in $\Alg'$ corresponds then to the pointwise multiplication of functions
and  $\phi$ viewed as a function on $\Gamma$ corresponds to a positive (respectively, unital)
functional on $\Alg$ if and only if it is positive definite (respectively, $\phi(e)=1$). This
implies that $\phi$ corresponds to an idempotent state if and only if it is an indicator function
(of a certain subset $S \subset \Gamma$) which is positive definite. It is a well known fact that
this happens if and only if $S$ is a subgroup of $\Gamma$ (\cite{Hewitt}, Cor. (32.7) and Example
(34.3 a)). It remains to prove that if $S$ is a subgroup of $\Gamma$ then the indicator function
$\chi_S$ is a Haar idempotent if and only if $S$ is normal. For the `if' direction assume that $S$
is a normal subgroup and consider the finite quantum group $\Blg=C^*(\Gamma/S)$. Define $j:\Alg \to
\Blg$ by
\[ j(f)= \sum_{\gamma \in \Gamma}  \alpha_{\gamma}\lambda_{[\gamma]}, \]
where $f= \sum_{\gamma \in \Gamma} \alpha_{\gamma} \lambda_{\gamma}$.  So-defined $j$ is a
surjective unital $^*$-homomorphism (onto $\Blg$).
 As the Haar state on $\Blg$ is given by
\[h_{\Blg} \left(\sum_{\kappa \in \Gamma/S}  \alpha_{\kappa} \lambda_{\kappa} \right) = \alpha_{[e]}, \]
there is
\[ h_{\Blg} (j(f)) = \sum_{\gamma \in S} \alpha_{\gamma},\]
so that $h_{\Blg} \circ j$ corresponds via the identification of $\Alg'$ with the functions on
$\Gamma$ to the characteristic function of $S$.

    Suppose now that $S$ is a subgroup of $\Gamma$
which is not normal and let $\gamma_0\in \Gamma$, $s_0 \in S$ be such that $\gamma_0 s_0 \gamma_0
^{-1} \notin S $. Denote by $\phi_S$ the state on $\Alg$ corresponding to the indicator function of
$S$. Define $f \in \Alg$ by $f= \lambda_{\gamma_0 s_0} - \lambda_{\gamma_0}$. Then
\[ f^*f = 2 \lambda_e - \lambda_{s_0^{-1}} - \lambda_{s_0}, \;\;\;\; f f^* =  2 \lambda_e - \lambda_{\gamma_0 s_0^{-1}\gamma_0^{-1}} - \lambda_{\gamma_0 s_0 \gamma_0^{-1}}.\]
This implies that
\[ \phi_S(f^*f) = 0,\;\;\; \phi_S(f f^*) = 2,\]
so that $N_{\phi_S}$ is not selfadjoint and $\phi_S$ must be non-Haar.
\end{proof}

\begin{cor}
Let $\Alg$ be a  cocommutative finite quantum group. The following are equivalent:
\begin{rlist}
\item there are no non-Haar idempotent states on $\Alg$;
\item $\alg\cong C^*(\Gamma)$ for a hamiltonian finite group $\Gamma$.
\end{rlist}
\end{cor}

This implies that the simplest example (or, to be precise, the example of the lowest dimension) of a compact quantum group on
which non-Haar idempotent states exist is a $C^*$-algebra of the permutation group $S_3$. One can give precise formulas:
$C^*(S_3)$ is isomorphic as a $C^*$-algebra  to $\bc \oplus \bc \oplus M_2(\bc)$, both  the coproduct and non-Haar idempotent
states may be explicitly described in this picture. The fact that this is indeed an example of the smallest dimension possible
may be deduced from the following statements: the smallest dimension of the quantum group which is neither commutative nor
cocommutative is 8 (the example is given by the Kac-Paljutkin quantum group, see the section below); there are no non-Haar
idempotents in the commutative case;  a group which is not hamiltonian has to have at least 6 elements (as all subgroups of index
2 are normal). By tensoring the algebra $C^*(S_3)$ with arbitrary infinite-dimensional compact quantum group $\alg$ and
considering a tensor product of a non-Haar idempotent state on $C^*(S_3)$ with the Haar state on $\alg$ we obtain examples of
idempotent states on a compact quantum group which do not arise as the Haar states on a quantum subgroup. There exist however
genuinely quantum (i.e.\ neither commutative nor cocommutative) compact quantum groups on which every idempotent state arises as
Haar state on a quantum subgroup - in particular in \cite{FST} it is shown that this is the case for $U_q(2)$ and $SU_q(2)$ ($q
\in (-1,1]$).

One may ask what are the quantum hypergroups arising via the construction in Theorem \ref{main} from non-Haar idempotent states
on $C^*(\Gamma)$. Let then $\phi: C^*(\Gamma) \to \bc$ be a non-Haar idempotent state, given by $S$, a (necessarily not normal)
subgroup  of $\Gamma$. A simple analysis shows that $\phi$ is the Haar state on the finite quantum hypergroup dual to the
commutative quantum hypergroup of functions on $\Gamma$ constant on the double cosets of $S$. We refer to \cite{qalghyp} for
explicit formulas.

The order relation in this case is determined by the formula
\[
\chi_{\Gamma_1} \prec \chi_{\Gamma_2} \quad\mbox{ if and only if} \quad \Gamma_2\subseteq \Gamma_1,
\]
and
\begin{eqnarray*}
\chi_{\Gamma_1}\vee\chi_{\Gamma_2} &=& \chi_{\Gamma_1}\chi_{\Gamma_2} = \chi_{\Gamma_1\wedge \Gamma_2}, \\
\chi_{\Gamma_1}\wedge\chi_{\Gamma_2} &=& \chi_{\Gamma_1\vee \Gamma_2}.
\end{eqnarray*}

\subsection*{Sekine quantum groups and examples of Pal type}

In \cite{Sekine} Y.\,Sekine introduced a family of finite quantum groups $\Alg_k$ ($k \in \bn$)
arising as bicrossed products of classical cyclic groups $\bz_k$: $\Alg_2$ is a celebrated
Kac-Paljutkin quantum group. All Sekine's quantum groups ($k \geq 2$) are neither commutative nor
cocommutative. Below we characterise for a given $k$ all quantum subgroups  of $\Alg_k$ and exhibit
for each $k \geq 2$ examples of idempotent states on $\Alg_k$ which are not Haar states on
subgroups.

Fix $k \in \bn$. let $\eta$ be a primitive $k$-th root of 1, and let $\bz_k:=\{0,1,\ldots,k-1\}$
denote the singly generated cyclic group of order $k$ (it is enough to rememember that the addition
in $\bz_k$ is understood modulo $k$).
 Set
 \[\Alg_k = \bigoplus_{i,j\in \bz_k} \bc d_{i,j} \oplus M_k (\bc).\]
The matrix units in $M_k(\bc)$ will be denoted by $\eij$ ($i,j =1,\ldots,k$). The coproduct in
$\alg_k$ is given by the following formulas:
\begin{equation} \label{com1} \Com(\dij) = \sum_{m,n \in \bzk} \left(d_{m,n} \ot d_{i-m,j-n}\right) +
\frac{1}{k} \sum_{m,n=1}^k \left( \eta^{i(m-n)} e_{m,n} \ot e_{m+j, n+j} \right)\end{equation} ($i,j \in \bzk$),
\begin{equation} \label{com2} \Com (\eij) = \sum_{m,n \in \bzk} \left( d_{-m,-n} \ot \eta^{m(i-j)} e_{i-n,j-n}
\right) +
                \sum_{m,n \in \bzk} \left(   \eta^{m(j-i)} e_{i-n,j-n}  \ot d_{m,n} \right)\end{equation}
($i,j\in\{1,\ldots,k\}$). As we are interested in the convolution of functionals, introduce the
dual basis in ${\Alg_k}'$ by
\[\wt{d}_{i,j} (d_{m,n}) = \delta_i^m \delta_j^n, \;\;  \wt{d}_{i,j} (e_{r,s}) = 0\]
($i,j,m,n \in \bzk, r,s\in\{1,\ldots,k\}$),
\[ \wt{e}_{i,j} (e_{r,s}) = \delta_i^r \delta_j^s, \;\;\; \wt{e}_{i,j} (d_{m,n})=0\]
($i,j,r,s\in\{1,\ldots,k\}, m,n \in \bzk $).

This leads to the following convolution formulas:
\[ \wt{d}_{i,j} \star \wt{d}_{m,n} = \wt{d}_{i+m,j+n},\]
($i,j,m,n \in \bzk$),
\[ \wt{d}_{i,j} \star \wt{e}_{r,s} = \eta^{i(s-r)} \wt{e}_{r-j, s-j},\]
($i,j \in \bzk, r,s \in \{1, \ldots,k\}$),
\[  \wt{e}_{r,s} \star \wt{d}_{i,j} = \eta^{i(s-r)} \wt{e}_{r+j, s+j},\]
($i,j \in \bzk, r,s \in \{1, \ldots,k\}$),
\[  \wt{e}_{i,j} \star \wt{e}_{r,s} = \delta_{r-i}^{s-j} \frac{1}{k} \sum_{p\in \bzk} \eta^{p(i-j)} \wt{d}_{p,r-i}\]
($i,j,r,s \in \{1, \ldots,k\}$). Putting all this together we obtain the following: if $\mu, \nu
\in {\Alg_k}'$ are given by
\[ \mu = \sum_{i,j \in \bzk} \alpha_{i,j} \wt{d}_{i,j} + \sum_{r,s\in \{1, \ldots,k\}} \kappa_{r,s} \wt{e}_{r,s},\]
\[ \nu = \sum_{i,j \in \bzk} \beta_{i,j} \wt{d}_{i,j} + \sum_{r,s\in \{1, \ldots,k\}} \omega_{r,s} \wt{e}_{r,s},\]
then
\[ \mu \star \nu =
\sum_{i,j \in \bzk} \gamma_{i,j} \wt{d}_{i,j} + \sum_{r,s\in \{1, \ldots,k\}} \theta_{r,s}
\wt{e}_{r,s},\] with
\[\gamma_{i,j} = \sum_{m,n \in \bzk} \alpha_{m,n} \beta_{i-m, j-n} + \frac{1}{k} \sum_{r,s \in \{1, \ldots,k\}}
\eta^{i(r-s)}\kappa_{r,s} \omega_{j+r, j+s},  \]
\[\theta_{r,s} =
\sum_{i,j \in \bzk} \eta^{i(s-r)} \left( \alpha_{i,j} \omega_{r+j, s+j} + \kappa_{r-j, s-j} \beta_{i,j}\right).\]

The following lemma is essentially equivalent to Lemma 2 in \cite{Sekine} (apparent differences
follow from the fact that we use a different basis for our functionals).

\begin{lem} \label{idemp2}
Let $\mu\in {\Alg_k}'$ be given by
\[ \mu = \sum_{i,j \in \bzk} \alpha_{i,j} \wt{d}_{i,j} + \sum_{r,s\in \{1, \ldots,k\}} \kappa_{r,s} \wt{e}_{r,s}.\]
Then $\mu$ is positive if and only if $ \alpha_{i,j} \geq 0$ and the matrix
$(\kappa_{r,s})_{r,s=1}^k$ is positive; $\mu(1) = 1$ if and only if $\sum_{i,j \in \bzk}
\alpha_{i,j}  +\sum_{r=1}^k \kappa_{r,r} = 1$; finally $\mu$ is an idempotent state if the
conditions above hold and
\[ \alpha_{i,j} = \sum_{m,n \in \bzk} \alpha_{m,n} \alpha_{i-m, j-n} + \frac{1}{k} \sum_{r,s \in \{1, \ldots,k\}}
\eta^{i(r-s)}\kappa_{r,s} \kappa_{j+r, j+s},\]
\[ \kappa_{r,s} =
\sum_{i,j \in \bzk} \eta^{i(s-r)} \alpha_{i,j} (\kappa_{r+j, s+j} + \kappa_{r-j, s-j}).\]
\end{lem}
\begin{proof}
For the first fact note that although the duality we use involves the transpose when compared to
the duality on $M_k(\bc)$ associated with the trace, the result remains valid, as a matrix is
positive if and only if its transpose is. The rest is straightforward.
\end{proof}

Before we use the formulas above to provide examples of non-Haar idempotent states on $\Alg_k$ ($k \geq 2$), let us characterise
the quantum subgroups of $\Alg_k$. Suppose that $\Blg$ is a $C^*$-algebra and $j:\Alg_k \to \Blg$ is a surjective unital
$^*$-homomorphism. It is immediate that $\Blg$ has to have a form $\bigoplus_{(i,j) \in S} \bc d_{i,j} \oplus' M_k(\bc)$, where
$S$ is a subset of $\bzk \times \bzk$ and the $'$ means that direct sum may or may not contain the $M_k(\bc)$ factor. The
respective $j$ have to be equal to identity on relevant factors in the direct sum decomposition of $\Alg_k$ and vanish on the
rest of them. Observe now that the co-morphism property of $j$ implies that the $\Com (\Ker\, j) \subset \Ker (j \ot j)$. Due to
the simple form of $j$ we actually have $\Ker (j \ot j) = (\Ker\, j \ot \Alg_k) + (\Alg_k \ot \Ker\, j)$ and the kernel admits an
easy interpretation on the level of subsets of $\bzk \times \bzk$. This allows us to prove the following.

\begin{tw} \label{subg}
Suppose that $\Blg$ is a quantum subgroup of $\Alg_k$. Then either $\Blg=\Alg_k$, or $\Blg \cong
C(\Gamma)$, where $\Gamma$ is a subgroup of $\bzk \times \bzk$. The Haar state on $\Alg_k$ is given
by the formula
\[ h_{\Alg_k} = \frac{1}{2k^2} \sum_{i,j \in \bzk} \wt{d}_{i,j} + \frac{1}{2k} \sum_{i=1}^{k} \wt{e}_{i,i},\]
and the Haar state on a quantum subgroup $C(\Gamma)$ of $\Alg_k$ is given by
\[ h_{\Gamma} = \frac{1}{\# \Gamma}\sum_{(i,j) \in \Gamma} \wt{d}_{i,j} \]
\end{tw}
\begin{proof}
By the discussion before the theorem we can assume that one of the following hold:
\begin{rlist}
\item $\Blg=\bigoplus_{(i,j) \in S} \bc d_{i,j} \oplus M_k(\bc)$,
\item $\Blg = \bigoplus_{(i,j) \in S} \bc d_{i,j}$,
\end{rlist}
where in both cases $S$ is a certain subset of $\bzk \times \bzk$, and if $j$ denotes the
corresponding surjective $^*$-homomorphism then
\begin{equation}\label{kincl}\Com (\Ker\, j) \subset \Ker (j \ot j) = (\Ker\, j \ot \Alg_k) + (\Alg_k \ot \Ker\, j).
\end{equation}
Denote $S' = \bzk \times \bzk \setminus S$. Consider first the case (i). Then the kernel of $j$ is
equal to $\bigoplus_{(i,j) \in S'} \bc d_{i,j}$. If $S'$ was nonempty, then by \eqref{kincl} and
the formula \eqref{com1} $\Ker\, j$ would have to have a nontrivial intersection with the
$M_k(\bc)$, which yields a contradiction. Therefore $S'=\emptyset$ and $\Blg = \Alg_k$.

Consider now the case (ii). Then $\Ker\, j \supset M_k(\bc)$ and therefore we can use again \eqref{kincl} and \eqref{com1} to
deduce the following: for every $(i,j) \in S'$ and $(m,n) \in \bzk \times \bzk$ either $(m,n)\in S'$ or $(i-m,j-n)\in S'$. This
is equivalent to stating that $S'S^{-1} \subset S'$. The latter implies that $S$ is a subsemigroup of $\bzk \times \bzk$; but as
the latter is a direct sum of the cyclic groups, every element is of finite order, so in fact $S$ must be a subgroup, denoted
further by $\Gamma$. This means that $\Blg = \bigoplus_{(i,j) \in \Gamma} \bc d_{i,j} \cong C(\Gamma)$. It is easy to check that
the $^*$-homomorphism $j$ in this case satisfies the condition \eqref{kincl}, so we are finished.

The formulas for the Haar states on subgroups are elementary to obtain; the Haar state on $\Alg_k$
was in fact computed in \cite{Sekine}.
\end{proof}

In the next proposition we exhibit the existence of non-Haar idempotent states on $\Alg_k$:

\begin{propn} \label{atyp}
Let $k\geq 2$. For each $l\in \{1,\ldots,k\}$ the state $\phi_l\in {\Alg_k}'$ given by
\[ \phi_l = \frac{1}{2k}\sum_{i \in \bzk} \wt{d}_{i,0} + \frac{1}{2}\wt{e}_{l,l}\]
is a non-Haar idempotent.
\end{propn}
\begin{proof}
The fact that each $\phi_l$ is idempotent follows from the conditions listed in Lemma \ref{idemp2};
it is also clear that none of the above states features in the complete list of Haar states on
subgroups of $\Alg_k$ listed in Theorem \ref{subg}.
\end{proof}

In the case $k=2$ the non-Haar idempotents above are the ones discovered by A.Pal in \cite{Pal}. In
general for $k\geq 2$ and $l\in\{1, \ldots,k\}$ one can show that, exactly as for the examples
treated in Theorem \ref{cocom}, $ N_{\phi_l} $ is not a selfadjoint subset of $\Alg_k$.

Theorem \ref{main} implies that each idempotent state on a finite quantum group arises, in a canonical way, as the Haar state on
a quantum subhypergroup. In the case described above we can compute explicitly the associated finite quantum hypergroups.

\begin{propn}
Let $k \geq 2$. Let $\Blg_{k}$ be the $C^*$-algebra of functions on the finite set containing $k+1$ distinct objects, with a
given family of minimal projections denoted by $p_j (j\in \Zk)$ and $q$. Define $\Com: \Blg_{k} \to \Blg_{k} \ot \Blg_{k}$ by the
linear extension of the following formulas:
\[ \Com(p_j) = \sum_{i \in \Zk} p_i \ot p_{j-i}+ \frac{1}{k}q \ot q ,\;\;\; j \in \Zk,\]
\[ \Com(q) =  \left(\sum_{i \in \Zk} p_i \right) \ot q + q \ot \left(\sum_{i \in \Zk} p_i \right).\]
The pair $(\Blg_k, \Com)$ is a finite quantum hypergroup.
\end{propn}

\begin{proof}
Straightforward computation. Note that the coproduct is explicitly seen to be positive, so also
completely positive, as $\Blg_k$ is commutative.
\end{proof}

As $\Blg_{k}$ is commutative and cocommutative, so has to be its dual. We compute it explicitly in the next proposition.

\begin{propn}
Let $k$ be as above. The coproduct on the dual quantum hypergroup of $\Blg_{k}$, denoted further by
$\Clg_{k}$, is given by the following formulas (the minimal projections are now denoted by $r_+,
r_-,r_1,\ldots,r_k)$:
\[ \hat{\Com} (r_m) = \sum_{\{n, j \in \{1, \ldots,k\} : n+j = m \textrm{ or }  n+j = m+k\}} r_n \ot r_j
 \;\;+  r_m \ot (r_+ + r_-) + (r_+ + r_-) \ot r_m, \]
\[ \hat{\Com} (r_+) = \frac{1}{2} \sum_{n=1}^k r_n \ot r_n + r_{+} \ot r_+ + r_- \ot r_-,\]
\[ \hat{\Com} (r_-) = \frac{1}{2} \sum_{n=1}^k r_n \ot r_n + r_{+} \ot r_- + r_- \ot r_+,\]
\end{propn}

\begin{proof}
Straightforward computation. In terms of the `dual' basis of ${\Blg_k}'$
\[ r_m = \sum_{j\in \Zk} \eta^{mj} \hat{p}_j + r_+ + r_-,\;\;\;\; m =1, \ldots,k,\]
\[ r_+ = \frac{1}{2k} \hat{p}_j + \frac{1}{2} \hat{q}, \;\; r_-\frac{1}{2k} \hat{p}_j - \frac{1}{2} \hat{q}.\]
\end{proof}

Note that $\Blg_2$ is isomorphic (in the quantum hypergroup category) to its dual. This is no
longer the case for $k>2$ (the same holds for quantum groups $\Alg_k$, see \cite{Sekine}).

The next proposition `explains' the origin of the non-Haar idempotents on Sekine's quantum groups and shows that each $\Alg_k$
contains at least $k$ distinct copies of $\Blg_k$.
\begin{propn}
Let $k \geq 2$ and $l\in \{1,\ldots,k\}$. The idempotent state $\phi_l\in {\Alg_k}'$ is the Haar
state on the quantum hypergroup $\Blg_{k}$.
\end{propn}

\begin{proof}
It is enough to define the map $\pi:\Alg_k \to \Blg_k$ by the linear extension of the formulas
\[\pi (d_{i,j}) = \delta_i^0 p_j, \;\; i, j \in \Zk,\]
\[ \pi (e_{r,s}) = \delta_r^l \delta_s^l q, \;\; r,s=1,\ldots,k,\]
observe that it intertwines respective comultiplications and it is completely positive as its
`matrix' part can be expressed as a composition of a compression to the diagonal and evaluation at
$l$-th coordinate.
\end{proof}

It follows from \cite{Pal} that for $k=2$ the list of non-Haar states on $\Alg_2$ in Proposition \ref{atyp} (and therefore the
list of idempotent states on $\Alg_2$ contained in Theorem \ref{subg} and in Proposition \ref{atyp}) is exhaustive. The analogous
result is no longer true for $ k\geq 4$. Indeed, fix $k \geq 4$ and let $p,m\in \bn$, $p,m \geq 2$ be such that $pm = k$. Then
the formulas in Lemma \ref{idemp2} imply that the functional $\gamma_{k,p}\in {\Alg_k}'$ given by
\[ \gamma_{k,p} = \frac{1}{4km}\sum_{i\in \bz_k} \sum_{l=0}^{m-1} \wt{d}_{i,lp} + \frac{1}{2m} \sum_{l=0}^{m-1}
\wt{e}_{lp, lp}\] is a (non-Haar) idempotent state on $\Alg_k$ which is different from the ones listed in Proposition \ref{atyp}.
As we do not know in general how all the idempotent states on $\Alg_k$ for $k\geq 2$ look like, we cannot describe the order
structure of $\mathcal{I}(\Alg_k)$. The order structure of  $\mathcal{I}(\Alg_2)$ was determined in \cite{UweRolf}.

\vspace*{0.5 cm}

\noindent \textbf{Acknowledgment}. Many of the ideas and techniques in the paper are inspired by
the work of Alfons Van Daele and his collaborators, which we gratefully acknowledge. In particular
the first named author would like to thank Alfons Van Daele for
discussions on group-like projections.
The paper was completed while U.F.\ was visiting the Graduate School of Information Sciences of Tohoku University as Marie-Curie
fellow. He would like to thank Professors Nobuaki Obata, Fumio Hiai, and the other members of the GSIS for their hospitality. We
are also very grateful to the anonymous referee for the careful reading of our
paper and pointing out a mistake in an earlier version.

\end{document}